\documentclass[a4paper, 11pt]{amsart}
\usepackage[francais]{babel}
\usepackage[T1]{fontenc}
\RequirePackage{xspace,amssymb,url,graphicx,hyperref,upref}
\usepackage[utf8]{inputenc}
\usepackage{mathrsfs}
\let\mathcal\mathscr

\AtBeginDocument{%
}

\theoremstyle{plain}
\newtheorem{thm}{Théorème}[section]
\newtheorem{prop}[thm]{Proposition}
\newtheorem{conj}[thm]{Conjecture}

\theoremstyle{remark}

\theoremstyle{definition}
\newtheorem{defi}[thm]{Définition}

\newenvironment{enumeratea}
{\bgroup\begin{enumerate}}
{\end{enumerate}\egroup}

\let\sep\mid
\def\siecle#1{\textsc{\romannumeral #1}\textsuperscript{e}~siècle}
\def\mto{\mathchoice{\longmapsto}{\mapsto}{\mapsto}{\mapsto}}
\def\sfrac#1#2{{#1}/{#2}}

\newcommand{\QQ}{\mathbb{Q}}

\makeatletter
\providecommand*{\shuffle}{%
  \mathbin{\mathpalette\shuffle@{}}%
}
\newcommand*{\shuffle@}[2]{%
  \sbox0{$#1\vcenter{}$}%
  \kern .15\ht0 
  \rlap{\vrule height .25\ht0 depth 0pt width 2.5\ht0}%
  \raise.1\ht0\hbox to 2.5\ht0{%
    \vrule height 1.75\ht0 depth -.1\ht0 width .17\ht0 %
    \hfill
    \vrule height 1.75\ht0 depth -.1\ht0 width .17\ht0 %
    \hfill
    \vrule height 1.75\ht0 depth -.1\ht0 width .17\ht0 %
  }%
  \kern .15\ht0 
}
\makeatother

\begin{document}
\title{Valeurs zêta multiples}
\author{Clément Dupont}
\address{Institut Montpelliérain Alexander Grothendieck (IMAG),
Université de Montpellier\\
CNRS\\
Place Eugène Bataillon\\
34090 Montpellier\\ France}
\email{clement.dupont@umontpellier.fr}
\urladdr{https://imag.umontpellier.fr/~dupont/}

\thanks{Cet article de survol accompagne deux exposés donnés à l'occasion des Journées mathématiques X-UPS 2019. L'auteur a bénéficié du soutien de la bourse ANR-18-CE40-0017 \og PERGAMO\fg de l'Agence Nationale de la Recherche}

\begin{abstract}
Les valeurs zêta multiples forment une famille de cons\-tantes mathématiques fondamentales qui contient notamment les valeurs aux entiers de la fonction zêta de Riemann. Si Euler leur a consacré des travaux, c’est seulement à la fin du \textsc{xx}\ieme siècle que mathématiciens et physiciens ont réalisé l’importance de ces nombres qui apparaissent naturellement dans des situations variées. Des travaux récents (Goncharov, Deligne, Brown, \ldots) ont mis en évidence une structure cachée qui est révélée par la géométrie: une théorie de Galois des valeurs zêta multiples. L’étude de cette structure a permis de prouver des résultats spectaculaires sur les relations algé\-briques satisfaites par ces nombres, que nous présenterons. Nous discuterons enfin de l’apparition des valeurs zêta multiples en physique des particules à travers le calcul d’intégrales de Feynman.
\end{abstract}
\maketitle
\thispagestyle{empty}

\tableofcontents

\section{Introduction}

Les \emph{valeurs zêta multiples} sont les sommes des séries multiples
\begin{equation}\label{eq:defiMZV}
\zeta(n_1,\ldots,n_r)=\sum_{1\leq k_1<\cdots <k_r}\frac{1}{k_1^{n_1}\cdots k_r^{n_r}}
\end{equation}
où les $n_i$ sont des nombres entiers supérieurs ou égaux à $1$ et $n_r\geq 2$, cette dernière condition assurant la convergence de la série. Pour $r=1$ on retrouve les \emph{valeurs zêta simples}\vspace*{-3pt}\enlargethispage{\baselineskip}%
\begin{equation}\label{eq:defiSZV}
\zeta(n)=\sum_{k=1}^{\infty}\frac{1}{k^n}
\end{equation}
qui sont les valeurs aux entiers de la fonction zêta de Riemann. Ces nombres ont été étudiés, notamment par Euler, bien avant que Riemann étudie $\zeta(s)$ comme fonction d'une variable $s$ \emph{complexe}, et ses liens avec la répartition des nombres premiers (voir à ce sujet les actes des Journées X-UPS 2002 \cite{xupszeta}). Les valeurs zêta simples, et plus généralement les valeurs zêta multiples, ont une richesse arithmétique d'un autre ordre et jouent le rôle de \emph{constantes fondamentales}. Elles se manifestent en effet dans de nombreux domaines des mathématiques et de la physique, notamment par le biais des intégrales de Feynman qui gouvernent les interactions entre particules élémentaires.

Il y a donc un intérêt pratique à \emph{organiser} la famille des valeurs zêta multiples et à comprendre les liens qu'elles entretiennent entre elles et avec d'autres constantes. C'est cette question, réactivée il y a à peine 30 ans, que nous abordons dans un premier temps par l'expé\-rimentation en partant de l'étude entreprise par Euler des valeurs zêta simples. De manière surprenante, il en ressort des réponses incroya\-blement structurées, qui trouvent leurs explications dans le \hbox{domaine} de la géométrie algébrique. Les valeurs zêta multiples s'avèrent alors être les incarnations tangibles de phénomènes profonds liés à la \emph{théorie des motifs}, une théorie de l'obstruction universelle pour les variétés algébriques pensée par Grothendieck dans les années 1960. Nous nous efforcerons dans un second temps d'introduire ces idées \hbox{modernes} en mettant l'accent sur leurs implications arithmétiques.

Ces résultats sont un premier témoignage de la puissance de l'étude \emph{géométrique} de certains nombres transcendants: les \emph{périodes}, qui sont les valeurs des intégrales définies par des formules algébriques. Une idée centrale qu'on développera est que la théorie de Galois \hbox{classique}, qui organise la famille des nombres algébriques, devrait se géné\-raliser à toutes les périodes et révéler des symétries cachées dans les constantes mathématiques. On décrira des résultats récents qui montrent que ces symétries se révèlent aussi dans les constantes physiques qui apparaissent dans le calcul des intégrales de Feynman.

\subsubsection*{Remerciements} L'auteur souhaite remercier Michel Alessandri pour sa relecture attentive d'une première version de ce texte.

\section{Valeurs zêta simples}

\subsection{Le problème de Bâle}

Dans la deuxième moitié du \siecle{17} se développent des techniques permettant de calculer des sommes de séries infinies. La quête d'une formule de forme fermée pour $\zeta(2)$, c'est-à-dire la somme des inverses des carrés des nombres entiers, est connue sous le nom de \emph{problème de Bâle}, du nom de la ville des frères Bernoulli qui ont popu\-larisé le problème (pourtant posé dès 1644 par Mengoli, originaire de Bologne). C'est Euler, un autre bâlois, qui en donnera la solution en 1735 \cite{euler1740summis}:
$$\zeta(2)=\frac{\pi^2}{6}\cdot$$
Plus généralement, Euler montre que les valeurs zêta paires sont des multiples rationnels de puissances de $\pi$:
\begin{equation}\label{eq:solution bale}
\zeta(2n) = \frac{|B_{2n}|}{2(2n)!}\,(2\pi)^{2n},
\end{equation}
où les $B_{2n}$ sont les nombres de Bernoulli, définis par leur série génératrice exponentielle
$$\sum_{n\geq 0}B_n\,\frac{t^n}{n!} = \frac{t}{e^t-1}\cdot$$
Le première preuve donnée par Euler consiste à identifier les coefficients du développement en produit infini de la fonction sinus\footnote{Cette formule peut être prouvée par des méthodes d'analyse complexe \cite[Chap.\,15, Exer.\,4]{rudin}, dont Euler ne disposait pas.}:
$$\frac{\sin(\pi x)}{\pi x} = \prod_{k=1}^{\infty}\Bigl(1-\frac{x^2}{k^2}\Bigr).$$

Quelques années plus tôt, en 1731, Euler avait déjà accompli un exploit mathématique en calculant une valeur approchée de $\zeta(2)$ avec $20$ chiffres significatifs, qui lui a permis de vérifier sa solution au problème de Bâle\footnote{Stirling avait obtenu la même approximation par d'autres méthodes en 1730 \cite[p.\,55]{stirling}.}. La série définissant $\zeta(2)$ convergeant très lentement (son reste est équivalent à~$1/n$), elle donne lieu à des approximations médiocres de sa somme et Euler a dû inventer pour l'occasion ce qu'on appelle aujourd'hui la \emph{formule d'Euler-Maclaurin} pour estimer la suite du développement asymptotique \cite{euler1738summatione}. Il obtient l'approximation suivante, dont l'erreur est majorée par $(10\,n^{11})^{-1}$:
$$\sum_{k=1}^n\frac{1}{k^2} \sim \zeta(2) -\frac{1}{n} +\frac{1}{2n^2} -\frac{1}{6n^3} + \frac{1}{30 n^5} - \frac{1}{42 n^7} + \frac{1}{30 n^9} \cdot$$
En calculant seulement $n=100$ termes on obtient donc l'approximation souhaitée:
$$\zeta(2) \simeq 1.6449340668482264364724076,$$
qui est bien valable à $10^{-23}$ près:
$$\frac{\pi^2}{6}\simeq 1.6449340668482264364724152.$$

Notons pour conclure que la formule \eqref{eq:solution bale} implique que les valeurs zêta paires sont des nombres transcendants, c'est-à-dire ne sont racines d'aucun polynôme non nul à coefficients rationnels. Cela résulte en effet de la transcendance de $\pi$, qui a été prouvée par Lindemann en 1882 \cite{lindemann}.

\subsection{Valeurs zêta impaires}

La recherche de formules de forme fermée pour les valeurs zêta impaires $\zeta(2n+1)$ fut moins fructueuse. Il est aujourd'hui communément admis que, contrairement aux valeurs zêta paires, elles ne s'expriment \emph{probablement} pas de manière algébrique en fonction de $\pi$ et ne satis\-font même \emph{probablement} pas de relations algébriques entre elles: elles forment une nouvelle famille de constantes mathématiques. Plus explicitement, on a la conjecture suivante:

\begin{conj}\label{conj:zeta impairs}
Les nombres $\pi$ et $\zeta(2n+1)$, pour $n\geq 1$, sont algébrique\-ment indépendants sur $\QQ$.
\end{conj}

Cette conjecture implique notamment que les valeurs zêta impaires sont des nombres transcendants. Aucun progrès n'a été fait dans la direction de cette conjecture avant Apéry, qui a montré l'irrationalité de $\zeta(3)$ en 1979 (voir aussi la jolie preuve alternative par Beukers du théorème d'Apéry \cite{beukersapery}):

\begin{thm}[Apéry \cite{apery}]
$\zeta(3)\notin\mathbb{Q}$.
\end{thm}

En 2001, Ball et Rivoal ont montré qu'une infinité de valeurs zêta impaires sont irrationnelles. Plus précisément:

\begin{thm}[Ball-Rivoal \cite{rivoalcras, ballrivoal}]
Pour $n\geq 1$, les nombres $1$, $\zeta(3)$, $\zeta(5)$, $\zeta(7)$, $\ldots, \zeta(2n+1)$ engendrent un sous-$\mathbb{Q}$-espace vectoriel de $\mathbb{R}$ de dimension au moins $\frac{1}{3}\log(2n+1)$. En particulier, il existe une infinité de valeurs zêta impaires irrationnelles.
\end{thm}

Comme $\frac{1}{3}\log(2n+1)\geq 3$ pour $2n+1\geq 8105$, il existe une valeur zêta impaire $\mathbb{Q}$-linéairement indépendante de $1$ et $\zeta(3)$, et donc irrationnelle, parmi $\zeta(5), \zeta(7),\ldots, \zeta(8105)$. Les bornes plus précises de Ball-Rivoal permettent d'arrêter cette liste à $\zeta(169)$. Le record de ce genre de résultat est détenu par Zudilin:

\begin{thm}[Zudilin \cite{zudilinfour}]
Il existe au moins un irrationnel parmi les nombres $\zeta(5)$, $\zeta(7)$, $\zeta(9)$, $\zeta(11)$.
\end{thm}

Nous renvoyons le lecteur au texte de Tanguy Rivoal dans ce volu\-me pour une discussion des techniques employées dans les preuves de tels énoncés, et à l'exposé au séminaire Bourbaki de Fischler \cite{fischlerbourbaki} pour plus de détails. L'irrationalité de $\zeta(5)$ ou de $\zeta(3)/\pi^3$, ou encore la transcendance de $\zeta(3)$, sont des problèmes ouverts.

C'est l'expérimentation qui permet de donner du poids à la conjecture \ref{conj:zeta impairs}. Une relation algébrique entre des nombres réels étant une relation linéaire entre des monômes formés de ces nombres, la conjecture revient à montrer qu'il n'existe aucune relation linéaire à coefficients entiers entre ces monômes. Les algorithmes de détection de relations entières comme PSLQ \cite{pslq} permettent de trouver des combinaisons linéaires \og petites\fg de certains nombres réels calculés avec une précision suffisante. Ils peuvent être utilisés de deux manières:
\begin{enumerate}
\item pour découvrir des relations entre ces nombres, qu'on peut ensuite essayer de démontrer mathématiquement (par exemple, trouver le polynôme minimal d'un nombre algébrique calculé avec une précision suffisante) ;
\item pour se convaincre de l'inexistence de relations en donnant des bornes inférieures sur la taille (degré et valeur absolue des coefficients) de relations potentielles entre ces nombres.
\end{enumerate}
On trouvera une application de cette dernière stratégie au test de la transcendance de $\zeta(3)$ dans \cite{baileyferguson}. Nous verrons plus loin que la conjecture \ref{conj:zeta impairs} est impliquée par une conjecture très générale, la \emph{conjecture des périodes} de Grothendieck, ce qui constitue une autre source, cette fois plus théorique, de certitude.

\section{L'algèbre des valeurs zêta multiples}\label{sec:MZV}

\subsection{Des valeurs zêta simples aux valeurs zêta multiples}

La conjecture \ref{conj:zeta impairs} traite de relations \emph{algébriques} entre valeurs zêta simples et nécessite donc de considérer des produits de tels nombres. On remarque alors que ces produits peuvent s'exprimer comme des combinaisons linéaires de valeurs zêta multiples, par exemple:
\begin{equation}\label{eq:zeta m zeta n}
\zeta(m)\zeta(n) = \zeta(m,n)+\zeta(n,m)+\zeta(m+n).
\end{equation}
Cette égalité se prouve aisément en décomposant le domaine de sommation:
\begin{equation}\label{eq:decoupage quasi shuffle}
\biggl(\sum_{k=1}^\infty\frac{1}{k^m} \biggr) \biggl(\sum_{l=1}^\infty\frac{1}{l^n} \biggr)= \biggl( \sum_{1\leq k<l} +\sum_{1\leq l<k}+\sum_{1\leq k=l} \biggr) \frac{1}{k^ml^n}\cdot
\end{equation}
Ce calcul montre que l'étude des relations \emph{algébriques} entre valeurs zêta simples est contenue dans l'étude de relations \emph{linéaires} entre valeurs zêta multiples. Autre observation importante: le membre de droite de \eqref{eq:zeta m zeta n} fait intervenir trois valeurs zêta multiples dont les indices somment à $m+n$. Cela justifie la définition suivante:

\begin{defi}
Le \emph{poids} d'une valeur zêta multiple \eqref{eq:defiMZV} est la somme des indices $n_1+\cdots+n_r$.
\end{defi}

Le calcul \eqref{eq:zeta m zeta n} a la généralisation suivante qu'on conseille au lecteur de prouver en anticipant le paragraphe \ref{sec:double melange}:

\begin{prop}\label{prop: produit MZVs}
Le produit de deux valeurs zêta multiples de poids $m$ et $n$ peut s'écrire comme une combinaison $\mathbb{Z}$-linéaire de valeurs zêta multiples de poids $m+n$.
\end{prop}

Les $2^{n-2}$ valeurs zêta multiples de poids $n$ ne sont pas linéairement indépendantes sur $\QQ$, comme le montre l'égalité des deux valeurs zêta multiples de poids $3$, déjà connue d'Euler:
$$\zeta(1,2)=\zeta(3).$$
De même, les quatre valeurs zêta multiples de poids $4$ sont colinéaires sur~$\mathbb{Q}$:
$$\zeta(1,1,2)=\zeta(4) \; ,\; \zeta(1,3)=\frac{1}{4}\zeta(4)\; ,\; \zeta(2,2)=\frac{3}{4}\zeta(4).$$
Nous verrons au paragraphe \ref{sec:double melange} une manière systématique de produire des relations linéaires entre valeurs zêta multiples. Contentons-nous pour l'instant d'estimer le \emph{nombre de relations} entre valeurs zêta multiples d'un poids donné. Pour cela, introduisons le sous-$\QQ$-espace vectoriel $\mathcal{Z}_n\subset \mathbb{R}$ engendré par les valeurs zêta multiples de poids $n$. Il est commode de poser $\zeta(\,)=1$ dans le cas $r=n=0$, de sorte que $\mathcal{Z}_0=\QQ$. Soit $\mathcal{Z}$ le sous-$\QQ$-espace vectoriel de $\mathbb{R}$ engendré par \emph{toutes} les valeurs zêta multiples ; c'est la somme des $\mathcal{Z}_n$ pour $n\geq 0$. La proposition \ref{prop: produit MZVs} implique que $\mathcal{Z}$ est une sous-$\QQ$-algèbre de $\mathbb{R}$.

\begin{defi}
On appelle $\mathcal{Z}$ l'\emph{algèbre des valeurs zêta multiples}.
\end{defi}

C'est la structure de cette algèbre qu'on souhaite étudier afin d'\emph{organiser} les constantes mathématiques fondamentales que sont les valeurs zêta multiples.

\subsection{Dimensions des espaces de valeurs zêta multiples}

En calculant des approximations suffisamment précises de toutes les valeurs zêta multiples de poids $\leq 12$ et en s'appuyant sur un algorithme de détection de relations entières, Zagier est arrivé à la conjecture suivante:

\begin{conj}[Zagier \cite{zagiervalues}]\mbox{}\label{conj:zagier dim}
\begin{enumerate}
\item On a une décomposition en somme directe:
$$\mathcal{Z}=\bigoplus_{n\geq 0}\mathcal{Z}_n.$$
Dit autrement, $\mathcal{Z}$ est \emph{graduée} par le poids: il n'y a pas de relations $\QQ$\nobreakdash-linéaires non triviales entre valeurs zêta multiples de poids différents.
\item On a pour tout entier $n$:
$$\dim_\QQ(\mathcal{Z}_n)=d_n,$$
où $(d_n)$ est la suite définie par récurrence par
$$\begin{cases} d_0=1 \; ,\; d_1=0 \; ,\; d_2=1 \; ; \\ d_n = d_{n-2}+d_{n-3} & (n\geq 3).\end{cases}$$
\end{enumerate}
\end{conj}

La suite $(d_n)$ est une variante \og cubique\fg de la suite de Fibonacci et a un comportement asymptotique donné par $d_n\sim C\alpha^n$, où $C$ est une constante et $\alpha\simeq 1,3247$ est la racine réelle du polynôme caractéristique $X^3-X-1$. Elle est donc bien inférieure au nombre $2^{n-2}$ de valeurs zêta multiples en poids $n$, ce qui suggère un grand nombre de relations linéaires entre valeurs zêta multiples d'un poids donné.

\begin{figure}[htb]\tabcolsep4.3pt
\begin{center}
\begin{tabular}{ c|c c c c c c c c c c c c c c }
$n$ &$0$& $1$ & $2$ & $3$ & $4$ & $5$ & $6$ & $7$ & $8$ & $9$ & $10$ & $11$ & $12$ & $13$ \\
\hline
$d_n$ & $1$ & $0$ & $1$ & $1$ & $1$ & $2$ & $2$ & $3$ & $4$ & $5$ & $7$ & $9$ & $12$ & $16$\\
$2^{n-2}$ & $-$ & $-$ & $1$ & $2$ & $4$ & $8$ & $16$ & $32$ & $64$ & $128$ & $256$ & $512$ & $1024$ & $2048$ \\
\end{tabular}
\end{center}
\end{figure}

L'égalité $\dim_\QQ(\mathcal{Z}_n)=d_n$ n'est connue que pour $n\leq 4$. Pour $n=5$ on peut prouver (par exemple en utilisant les relations de double \hbox{mélange} présentées au paragraphe \ref{sec:double melange}) que les $8$ valeurs zêta multiples de poids $5$ sont toutes des combinaisons $\QQ$-linéaires de $\zeta(2,3)$ et $\zeta(3,2)$. On ne sait cependant pas si le quotient de ces deux nombres est irrationnel. 

Un autre enseignement de ces calculs expérimentaux est qu'il y a conjecturalement bien plus dans l'algèbre des valeurs zêta multiples que des polynômes en les valeurs zêta simples $\zeta(n)$, qui ne contribuent au comptage des dimensions que par un facteur sous-exponentiel. La~première occurrence d'une valeur zêta multiple qui ne s'exprime pas conjecturalement en fonction de valeurs zêta simples est en poids~$8$: l'expérimentation suggère par exemple que $\zeta(3,5)$ et $\zeta(5,3)$ ne sont pas des combinaisons $\QQ$-linéaires de $\zeta(2)^3$, $\zeta(3)^2$ et $\zeta(3)\zeta(5)$ --- même si leur somme l'est par \eqref{eq:zeta m zeta n}.

La borne supérieure dans la conjecture \ref{conj:zagier dim} a été prouvée indépendamment par Goncharov et Terasoma:

\begin{thm}[Goncharov \cite{goncharovmzvmodular}, Terasoma \cite{terasoma}]\label{thm:goncharov terasoma}
On a pour tout entier $n$: $$\dim_\QQ(\mathcal{Z}_n)\leq d_n.$$
\end{thm}

De manière surprenante, la preuve de ce résultat ne consiste pas à produire suffisamment de relations $\mathbb{Q}$-linéaires entre valeurs zêta multiples, et ignore même complètement l'existence de familles remar\-quables de telles relations ! Elle s'appuie plutôt sur des idées venant de la géométrie algébrique et que nous introduirons dans la suite de ce texte, voir le paragraphe \ref{subsec:MZV mot first}.

La conjecture suivante, due à Hoffman, propose une base conjecturale de l'espace des valeurs zêta multiples.

\begin{conj}[Hoffman \cite{hoffman}]\label{conj:hoffman} Les valeurs zêta multiples $\zeta(n_1,\ldots,n_r)$ avec $n_i\in\{2,3\}$ forment une base du $\QQ$-espace vectoriel~$\mathcal{Z}$.
\end{conj}

Elle implique la formule de dimension de Zagier (conjecture \ref{conj:zagier dim}) puisque~$d_n$ est le nombre de mots $(n_1,\ldots,n_r)$ de longueur quelconque formés de $2$ et de $3$ et satisfaisant à $n_1+\cdots+n_r=n$. Brown a montré que les \og valeurs zêta multiples de Hoffman\fg (celles dont les arguments sont des $2$ et des $3$) engendrent l'espace des valeurs zêta multiples:

\begin{thm}[Brown \cite{brownMTMZ}]\label{thm:brown}
Les valeurs zêta multiples $\zeta(n_1,\ldots,n_r)$ avec $n_i\in\{2,3\}$ engen\-drent le $\mathbb{Q}$-espace vectoriel $\mathcal{Z}$.
\end{thm}

À l'instar du théorème \ref{thm:goncharov terasoma} de Goncharov et Terasoma, ce résultat est prouvé par des techniques sophistiquées issues de la géométrie algé\-brique. Remarquons que le théorème de Brown n'est pas effectif: il ne donne pas de formule, ni même d'algorithme, pour écrire une valeur zêta multiple comme combinaison $\QQ$-linéaire des valeurs zêta multiples de Hoffman. Cependant, les coefficients de telles combinaisons linéaires sont \og petits\fg en pratique et peuvent être découverts expérimentalement. Mieux, la preuve du théorème de Brown permet de découvrir ces coefficients \og un par un\fg en quelque sorte, c'est-à-dire en reconnaissant à chaque étape \emph{un seul} nombre rationnel connu avec une grande précision. Ce théorème est donc effectif \og en pratique\fg et donne un outil pour écrire les valeurs zêta multiples dans la base de Hoffman, ou dans toute autre base (voir \cite{browndecomposition}).

\subsection{La structure de l'algèbre des valeurs zêta multiples}

La structure d'algèbre de $\mathcal{Z}$ est conjecturalement très simple. Intro\-duisons des symboles formels $f_n$ pour $n\geq 3$ impair et l'espace
$$\mathcal{A}=\QQ\langle f_3,f_5,f_7,f_9,\ldots \rangle$$
des polynômes non commutatifs en ces symboles, dont une base est donné par les mots dans l'alphabet $\{f_3,f_5,f_7,f_9,\ldots\}$. On définit sur~$\mathcal{A}$ le \emph{produit de mélange} $\shuffle$ qui consiste comme son nom l'indi\-que à mélanger les mots comme deux paquets de cartes, en gardant l'ordre des lettres dans chaque mot\footnote{En anglais, mélange se dit \emph{shuffle}, d'où le choix de la lettre $\shuffle$ de l'alphabet cyrillique, qui se prononce \emph{sha}.}. Par exemple: $f_3\shuffle (f_5f_7)=f_3f_5f_7 + f_5f_3f_7+f_5f_7f_3$. La formule générale est:
\begin{equation}\label{eq:defi shuffle f}
(f_{i_1}\cdots f_{i_r}) \shuffle (f_{i_{r+1}}\cdots f_{i_{r+s}}) = \sum_{\sigma\in \shuffle(r,s)} f_{i_{\sigma^{-1}(1)}}\cdots f_{i_{\sigma^{-1}(r+s)}},
\end{equation}
où $\shuffle(r,s)$ est l'ensemble des permutations $\sigma\in\mathfrak{S}_{r+s}$ qui sont strictement croissantes sur les $r$ premiers indices et les $s$ derniers indices. Ce~produit fait de l'espace $\mathcal{A}$ une algèbre \emph{commutative} \cite{reutenauerfreelie}. Enfin, on introduit un nouveau symbole formel $f_2$ et on considère l'algèbre
$$\mathcal{F}=\mathcal{A}[f_2]$$
des polynômes en $f_2$ avec coefficients dans l'algèbre de mélange $\mathcal{A}$. Une base de $\mathcal{F}$ est donnée par les monômes $f_2^kf_{i_1}\cdots f_{i_r}$ avec les $i_j\geq 3$ impairs. On introduit un \emph{degré} sur $\mathcal{F}$ en décrétant qu'un symbole $f_n$ a degré $n$ pour $n=2$ ou $n\geq 3$ impair. Cela fait de $\mathcal{F}$ une algèbre graduée.

\begin{conj}\label{conj:structure algebre Z}
L'algèbre $\mathcal{Z}$ est graduée par le poids et on a un isomorphisme d'algèbres graduées:
$$\mathcal{Z} \simeq \mathcal{F}.$$
\end{conj}

Un isomorphisme prédit par la conjecture n'a pas de raison d'être unique, mais il est naturel de le normaliser afin que $\zeta(n)$ corresponde au symbole $f_n$. Notons que cette conjecture implique la conjecture~\ref{conj:zagier dim} de Zagier puisque le nombre de monômes de poids $n$ dans $\mathcal{F}$ est exactement $d_n$. En effet, la série génératrice correspondante est
$$\sum_{n\geq 0}\dim_\QQ(\mathcal{F}_n)\, t^n = \frac{1}{1-t^2}\, \frac{1}{1-(t^3+t^5+t^7+t^9+\cdots)} = \frac{1}{1-t^2-t^3},$$
où l'on reconnaît la série génératrice de la suite $(d_n)$.

La conjecture \ref{conj:structure algebre Z} pourrait faire croire à une base \og simple\fg de l'algèbre des valeurs zêta multiples, qui correspondrait à la base des monômes de $\mathcal{F}$ et serait donnée par les
$$\zeta(2)^k\, \zeta(i_1,\ldots,i_r)$$
avec les $i_j\geq 3$ impairs. Mais ces éléments ne sont pas linéairement indépen\-dants, comme le montre la relation suivante en poids $12$:
$$28\,\zeta(3,9) + 150\,\zeta(5,7) + 168 \, \zeta(7,5) = \frac{5197}{691}\zeta(12) = \frac{332\, 608}{875\, 875} \, \zeta(2)^6.$$
Cette relation est la première d'une famille infinie de relations découverte par Gangl-Kaneko-Zagier et indexée par les formes modulaires cuspidales pour le groupe modulaire $\mathrm{SL}_2(\mathbb{Z})$ \cite{ganglkanekozagier}.

\subsection{Valeurs zêta multiples motiviques et théorie de Galois}\label{subsec:MZV mot first}

Un des ingrédients principaux dans les preuves des théorèmes \ref{thm:goncharov terasoma} et~\ref{thm:brown} est la notion de \emph{valeur zêta multiple motivique} que nous décrivons dès maintenant afin d'éclairer les liens entre les conjectures et les résultats dont il a été question jusqu'ici --- pour plus de détails, voir les paragraphes \ref{sec:motifs}, \ref{sec:pi1mot}, \ref{sec:applis pi1mot}. Les valeurs zêta multiples motiviques sont des éléments
$$\zeta^{\mathrm{mot}}(n_1,\ldots,n_r)\;\in\; \mathcal{P}^{\mathrm{mot}}$$
d'une $\mathbb{Q}$-algèbre $\mathcal{P}^{\mathrm{mot}}$ dont les éléments sont appelés \emph{périodes motiviques}. Cette algèbre n'est pas constituée de nombres à proprement parler mais possède un morphisme de $\mathbb{Q}$-algèbres
$$\mathrm{per}:\mathcal{P}^{\mathrm{mot}}\to \mathbb{C},$$
appelé le \emph{morphisme de période}, qui est tel que
$$\mathrm{per}(\zeta^{\mathrm{mot}}(n_1,\ldots,n_r))=\zeta(n_1,\ldots,n_r).$$
Les valeurs zêta multiples motiviques sont donc vues comme un \emph{relè\-vement} des valeurs zêta multiples \og ordinaires\fg. Une conjecture centrale, due à Grothendieck, veut que ces deux versions contiennent exactement la même information:

\begin{conj}[Conjecture des périodes, cas particulier]\label{conj:periodes}
Le morphisme $\mathrm{per}$ est injectif.
\end{conj}

Cette conjecture affirme que toute relation algébrique entre valeurs zêta multiples est aussi satisfaite par les valeurs zêta multiples motiviques: il ne devrait donc pas y avoir de différence entre l'algè\-bre~$\mathcal{Z}$ et la sous-algèbre $\mathcal{Z}^{\mathrm{mot}}\subset \mathcal{P}^{\mathrm{mot}}$ engendrée par les valeurs zêta multiples motiviques. Une grande force des idées motiviques est que contrairement à $\mathcal{Z}$, on sait dire beaucoup de choses sur la structure de $\mathcal{Z}^{\mathrm{mot}}$:
\begin{enumeratea}
\item L'algèbre $\mathcal{P}^{\mathrm{mot}}$ et sa sous-algèbre $\mathcal{Z}^{\mathrm{mot}}$ sont naturellement graduées et $\zeta^{\mathrm{mot}}(n_1,\ldots,n_r)$ est dans la partie homogène de degré $n=n_1+\cdots +n_r$.
\item On a un isomorphisme d'algèbres graduées $\mathcal{P}^{\mathrm{mot}}\simeq \mathcal{F}$, qui peut être choisi de telle sorte à faire correspondre $\zeta^{\mathrm{mot}}(n)$ avec $f_n$ pour $n=2$ ou $n\geq 3$ impair.
\end{enumeratea}
La partie (a) implique que la deuxième partie de la conjecture \ref{conj:zagier dim} de Zagier est connue pour les valeurs zêta multiples \emph{motiviques}. La partie (b) prouve le théorème \ref{thm:goncharov terasoma} de Goncharov et Terasoma grâce aux d'inégalités:
$$\dim_{\mathbb{Q}}(\mathcal{Z}_n)\leq \dim_{\mathbb{Q}}(\mathcal{Z}_n^{\mathrm{mot}}) \leq \dim_{\mathbb{Q}}(\mathcal{P}_n^{\mathrm{mot}}) = \dim_{\mathbb{Q}}(\mathcal{F}_n) = d_n.$$

Inspiré par la conjecture \ref{conj:hoffman} de Hoffman, Brown prouve:

\begin{thm}[Brown \cite{brownMTMZ}]\label{thm:brown bis}
Les valeurs zêta multiples motiviques $\zeta^{\mathrm{mot}}(n_1,\ldots,n_r)$ avec $n_i\in \{2,3\}$ sont linéairement indépendantes dans $\mathcal{P}^{\mathrm{mot}}\!$.
\end{thm}

Un argument de dimension montre alors que les valeurs zêta multiples motiviques de Hoffman forment une base de $\mathcal{Z}^{\mathrm{mot}}=\mathcal{P}^{\mathrm{mot}}$. En appliquant le morphisme de période on obtient donc le théorème~\ref{thm:brown}. De plus, l'algèbre graduée $\mathcal{Z}^{\mathrm{mot}}$ est isomorphe à l'algèbre abstraite~$\mathcal{F}$ et il résulte que la conjecture \ref{conj:zagier dim} de Zagier, la conjecture \ref{conj:hoffman} de Hoffman et la conjecture \ref{conj:structure algebre Z} sur la structure de l'algèbre $\mathcal{Z}$ sont toutes les trois satisfaites au niveau des valeurs zêta multiples motiviques. Elles sont donc logiquement équivalentes à la conjecture \ref{conj:periodes} des périodes, qui affirme que le morphisme des périodes induit un isomorphisme
$$\mathcal{Z}^{\mathrm{mot}} \stackrel{?}{\simeq} \mathcal{Z}.$$
Notons au passage que la conjecture des périodes implique la conjecture \ref{conj:zeta impairs} sur les valeurs zêta impaires puisque celle-ci est vraie pour les valeurs zêta motiviques (en effet, dans l'algèbre $\mathcal{F}$, les éléments $f_2$, $f_3$, $f_5$, $f_7$, etc. sont algébriquement indépendants). 

La preuve du théorème \ref{thm:brown bis} utilise une structure fine de l'algèbre des valeurs zêta multiples motiviques:

\begin{enumeratea}\setcounter{enumi}{2}
\item Il existe une action très riche d'un groupe, le \emph{groupe de Galois motivique}, sur l'algèbre $\mathcal{P}^\mathrm{mot}$, dont l'effet sur les valeurs zêta multiples motiviques $\zeta^{\mathrm{mot}}(n_1,\ldots,n_r)$ peut être calculé explicitement.
\end{enumeratea}
Cet ingrédient doit être vu comme une \emph{théorie de Galois} des valeurs zêta multiples motiviques, qui est analogue à la théorie de Galois classique des nombres algébriques, où les groupes de Galois finis sont remplacés par des groupes de matrices sur $\mathbb{Q}$. La différence est que le groupe de Galois motivique n'agit pas sur l'algèbre des valeurs zêta multiples $\mathcal{Z}$ mais sur sa variante motivique $\mathcal{Z}^{\mathrm{mot}}$ --- qui est la même chose si on croit à la conjecture \ref{conj:periodes}. Il s'agit d'une instance de l'idée générale d'une théorie de Galois des périodes, développée par André \cite{andregalois, andrebourbaki} d'après la philosophie des motifs de Grothendieck.

\section{Relations de double mélange}\label{sec:double melange}

Comme promis, nous décrivons ici une manière systématique et combinatoirement très riche de produire des relations linéaires (conjecturalement \emph{toutes} les relations linéaires) entre valeurs zêta multiples.

\subsection{Produit de quasi-mélange}

Commençons par expliciter la structure combinatoire derrière le calcul du produit \eqref{eq:zeta m zeta n} et la proposition \ref{prop: produit MZVs}. On considère sur l'espace
$$\mathcal{Y}=\QQ\langle y_1,y_2,y_3,\ldots\rangle$$ des polynômes non commutatifs en une infinité de variables formelles~$y_i$ un produit $*$, dit produit de \emph{quasi-mélange}, qui satisfait à $\alpha*1=1*\alpha=\alpha$ pour tout monôme $\alpha$, et la formule de récurrence (sur le degré des monômes):
$$(\alpha\,y_m)*(\beta\,y_n) = ((\alpha\,y_m)*\beta)\,y_n + (\alpha*(\beta\,y_n))\,y_m + (\alpha*\beta)\,y_{m+n},$$
où $\alpha$ et $\beta$ sont des monômes. Par exemple, on vérifie qu'on a la formule suivante, censée rappeler \eqref{eq:zeta m zeta n}:
$$y_m*y_n=y_my_n+y_ny_m+y_{m+n}.$$
Le produit $*$ consiste donc à mélanger les mots comme dans le produit de mélange de la section précédente, puis à fusionner certaines paires de lettres consécutives en ajoutant leurs indices\footnote{Dans la littérature anglo-saxonne on trouve parfois l'expression \emph{stuffle} pour le produit de quasi-mélange, qui est un mot-valise formé sur \emph{shuffle} (mélange) et \emph{stuff} (choses): le quasi-mélange, c'est \og le mélange, plus des choses\fg.}. Il fait de $\mathcal{Y}$ une $\mathbb{Q}$-algèbre commutative, appelée \emph{algèbre de quasi-mélange}, introduite par Hoffman \cite{hoffmanquasishuffle}.

Associons un monôme dans $\mathcal{Y}$ à un multi-indice par la correspondance:
$$\underline{n}=(n_1,\ldots,n_r) \quad\longleftrightarrow\quad Y_{\underline{n}}=y_{n_1}\cdots y_{n_r}.$$
Appelons $\underline{n}$ ou $Y_{\underline{n}}$ \emph{convergent} si $r=0$ ou $n_r\geq 2$. Les monômes convergents forment une base d'une sous-algèbre $\mathcal{Y}_0\subset\mathcal{Y}$. Écrivons les coefficients du produit de quasi-mélange dans cette base sous la forme:
$$Y_{\underline{m}}*Y_{\underline{n}}=\sum_{\underline{p}} \,v^{\underline{p}}_{\underline{m},\underline{n}}\,Y_{\underline{p}} \quad\quad (v^{\underline{p}}_{\underline{m},\underline{n}}\in\mathbb{N}).$$

\begin{prop}\label{prop:stuffle}
On a la formule de produit de quasi-mélange des valeurs zêta multiples, pour $\underline{m}$ et $\underline{n}$ convergents:
$$\zeta(\underline{m})\zeta(\underline{n}) = \sum_{\underline{p}}\,v^{\underline{p}}_{\underline{m},\underline{n}}\,\zeta(\underline{p}).$$
\end{prop}

\begin{proof}
La preuve de cette formule consiste à remarquer que le découpage du domaine de sommation dans un produit de séries multiples du type \eqref{eq:defiMZV} est précisément modélisé par le produit de quasi-mélange, comme dans le cas particulier \eqref{eq:decoupage quasi shuffle}.
\end{proof}

Dit autrement, l'application $Y_{\underline{n}}\mto \zeta(\underline{n})$ définit un morphisme d'algè\-bres de $(\mathcal{Y}_0,*)$ vers $\mathbb{R}$. On peut montrer qu'il existe une unique manière de l'étendre en un morphisme sur l'algèbre $(\mathcal{Y},*)$ tout entière qui s'annule sur $y_1$ (cela revient à poser $\zeta(1)=0$). Ce processus s'appelle la \emph{régularisation de quasi-mélange} des valeurs zêta multiples.

\subsection{Une formule intégrale pour les valeurs zêta multiples}

Si les valeurs zêta multiples sont définies comme des sommes de \emph{séries} multiples \eqref{eq:defiMZV}, il est avantageux de les interpréter comme des valeurs d'\emph{intégrales} multiples. Pour cela introduisons deux formes différentielles:
$$\omega_0(x) = \frac{dx}{x} \quad \mbox{ et }\quad \omega_1(x)=\frac{dx}{1-x}\cdot$$
À un mot $(a_1,\ldots,a_n)\in\{0,1\}^n$ on associe l'\emph{intégrale itérée}:
\begin{equation}\label{eq:defi int it 0 1}
\mathrm{I}(a_1,\ldots,a_n) = \int_{\Delta^n}\omega_{a_1}(t_1)\cdots \omega_{a_n}(t_n),
\end{equation}
où $\Delta^n=\{0\leq t_1\leq \cdots\leq t_n\leq 1\}$ est le simplexe standard de dimension $n$. On remarque que l'intégrale \eqref{eq:defi int it 0 1} converge exactement quand $a_1\neq 0$ et $a_n\neq 1$. La proposition suivante est attribuée à Kontsevich, où l'on utilise la notation $\{0\}^r$ pour une suite de $r$ zéros consécutifs.

\begin{prop}\label{prop:MZVs it int}
On a
$$\zeta(n_1,\ldots,n_r) = \mathrm{I}(1,\{0\}^{n_1-1},\ldots,1,\{0\}^{n_r-1}).$$
\end{prop}

\begin{proof}
On fait apparaître des séries géométriques dans l'intégrale itérée par le développement:
$$\omega_1(t)=\sum_{n\geq 0}t^n\, dt,$$
puis on échange les sommes et les intégrales, et on calcule les intégrales dans l'ordre croissant des indices. Par exemple:
\begin{align*}
\mathrm{I}(1,0) & = \iint_{0\leq x\leq y\leq 1} \frac{dx}{1-x}\frac{dy}{y}  = \sum_{k\geq 0}\int_0^1\left(\int_0^y x^k dx\right)\frac{dy}{y} \\
& = \sum_{k\geq 0} \frac{1}{k+1} \int_0^1 y^{k+1}\frac{dy}{y}  = \sum_{k\geq 0}\frac{1}{(k+1)^2} = \sum_{k=1}^\infty\frac{1}{k^2} = \zeta(2).\qedhere
\end{align*}
\end{proof}

Cette proposition montre que les valeurs zêta multiples sont des périodes au sens de Kontsevich-Zagier:

\begin{defi}[Kontsevich-Zagier \cite{kontsevichzagier}]\label{defi:periodeKZ}
Une \emph{période} (effective) est un nombre complexe dont les parties réelle et imaginaire peuvent s'écrire comme des intégrales absolument convergentes de fractions rationnelles sur $\mathbb{Q}$ sur des domaines de $\mathbb{R}^n$ définis par des inégalités entre polynômes sur $\mathbb{Q}$.
\end{defi}

Nous renvoyons le lecteur au texte de Javier Fresán dans ce volume pour une discussion de la notion de période.

Les valeurs zêta multiples ont d'autres représentations intégrales utiles: par exemple, le changement de variables $t_i=x_i\cdots x_n$, pour $i=1,\ldots,n$, permet d'écrire une intégrale itérée \eqref{eq:defi int it 0 1} comme une intégrale sur l'hypercube $[0,1]^n$, sur le modèle de:
$$\zeta(2) = \iint_{[0,1]^2}\frac{dx\, dy}{1-xy}\cdot$$

\subsection{Produit de mélange}\label{subsec:shuffle}

Les intégrales itérées \eqref{eq:defi int it 0 1} ont une structure multiplicative remarquable. Nous la modélisons de manière parallèle au produit de quasi-mélange en introduisant l'espace
$$\mathcal{X}=\QQ\langle x_0,x_1\rangle$$
des polynômes non commutatifs en deux variables $x_0,x_1$. Le \emph{produit de mélange} $\shuffle$ sur $\mathcal{X}$ a déjà été introduit plus haut dans un autre contexte \eqref{eq:defi shuffle f}, et fait de $\mathcal{X}$ une $\QQ$-algèbre commutative. Par exemple, on a:\vspace*{-3pt}\enlargethispage{.5\baselineskip}%
\begin{equation}\label{eq:example shuffle x0 x1}
x_1x_0\shuffle x_1x_0 = 2 \, x_1x_0x_1x_0 + 4\, x_1x_1x_0x_0.
\end{equation}
Inspirés par la proposition \ref{prop:MZVs it int}, associons un monôme dans $\mathcal{X}$ à un multi-indice par la correspondance:
$$\underline{n}=(n_1,\ldots,n_r) \quad\longleftrightarrow\quad X_{\underline{n}}=x_1x_0^{n_1-1}\cdots x_1x_0^{n_r-1}.$$
Les monômes $X_{\underline{n}}$ pour $\underline{n}$ convergent forment une base d'une sous-algèbre $\mathcal{X}_0\subset \mathcal{X}$. Écrivons les coefficients du produit de mélange dans cette base sous la forme:
$$X_{\underline{m}}\shuffle X_{\underline{n}} = \sum_{\underline{p}} \,u_{\underline{m},\underline{n}}^{\underline{p}}\, X_{\underline{p}},\vspace*{-3pt}$$
où les coefficients $u_{\underline{m},\underline{n}}^{\underline{p}}$ sont des entiers naturels.

\begin{prop}\label{prop:shuffle}
On a la formule de produit de mélange des valeurs zêta multiples, pour $\underline{m}$ et $\underline{n}$ convergents:
$$\zeta(\underline{m})\zeta(\underline{n}) = \sum_{\underline{p}}\,u^{\underline{p}}_{\underline{m},\underline{n}}\,\zeta(\underline{p}).$$
\end{prop}

\begin{proof}
On écrit les valeurs zêta multiples comme des intégrales itérées (proposition \ref{prop:MZVs it int}) et on utilise le théorème de Fubini pour calculer un produit de telles intégrales:
$$\mathrm{I}(a_1,\ldots,a_r)\,\mathrm{I}(a_{r+1},\ldots,a_{r+s})=\int_{\Delta^r\times\Delta^s}\omega_{a_1}(t_1)\cdots \omega_{a_{r+s}}(t_{r+s}).$$
Le produit de simplexes $\Delta^r\times\Delta^s$ est le polytope de $\mathbb{R}^{r+s}$ défini par les inégalités $0\leq t_1\leq\cdots \leq t_r\leq 1$ et $0\leq t_{r+1}\leq\cdots \leq t_{r+s}\leq 1$. Il~peut donc être triangulé par des simplexes
$$0\leq t_{\sigma(1)}\leq \cdots\leq t_{\sigma(r+s)}\leq 1,$$
où $\sigma\in \mathfrak{S}_{r+s}$ est une permutation qui est strictement croissante sur les $r$ premiers indices et sur les $s$ derniers indices, c'est-à-dire un élément de $\shuffle(r,s)$. Comme ces simplexes s'intersectent le long de sous-ensembles de mesure nulle, l'intégrale se décompose en
$$\mathrm{I}(a_1,\ldots,a_r)\,\mathrm{I}(a_{r+1},\ldots,a_{r+s}) =\sum_{\sigma\in \shuffle(r,s)}\mathrm{I}(a_{\sigma^{-1}(1)},\ldots,a_{\sigma^{-1}(r+s)}),$$
ce qui correspond bien au produit shuffle des monômes $X_{\underline{n}}$.
\end{proof}

Dit autrement, l'application $X_{\underline{n}}\mto \zeta(\underline{n})$ définit un morphisme d'algèbres de $(\mathcal{X}_0,\shuffle)$ vers $\mathbb{R}$. On peut montrer qu'il existe une unique manière de l'étendre en un morphisme sur l'algèbre $(\mathcal{X},\shuffle)$ tout entière qui s'annule en $x_0$ et $x_1$ (cela revient à poser $\zeta(1)=0$). Ce~processus s'appelle la \emph{régularisation de mélange} des valeurs zêta multiples.

\subsection{Relations de double mélange}

La comparaison des deux manières de calculer le produit de deux valeurs zêta multiples (propositions \ref{prop:stuffle} et \ref{prop:shuffle}) donne lieu à des relations linéaires:

\begin{prop}
Si $\underline{m}$ et $\underline{n}$ sont des multi-indices convergents on a l'égalité:
\begin{equation}\label{eq:double shuffle relation}
\sum_{\underline{p}}\,(u^{\underline{p}}_{\underline{m},\underline{n}}-\,v^{\underline{p}}_{\underline{m},\underline{n}})\,\zeta(\underline{p}) = 0.
\end{equation}
\end{prop}

Ces relations, dites \emph{relations de double mélange}, sont en général non triviales ; par exemple, pour $\underline{m}=\underline{n}=(2)$, l'égalité entre le produit de mélange \eqref{eq:example shuffle x0 x1} et le produit de quasi-mélange \eqref{eq:zeta m zeta n} donne:
$$4\, \zeta(1,3)=\zeta(4).$$

Si $\underline{m}=(1)$ et $\underline{n}$ est un multi-indice convergent, le lecteur remarquera que le membre de gauche de \eqref{eq:double shuffle relation} ne fait intervenir que des multi-indices convergents (les deux occurrences de la valeur zêta multiple divergente $\zeta(\underline{n},1)$ se compensent), et on peut montrer que la relation correspondante est bien satisfaite. Dans le cas $\underline{n}=(2)$ on retrouve l'égalité $\zeta(1,2)=\zeta(3)$. (Celle-ci peut aussi se déduire de la représentation de $\zeta(3)$ comme intégrale itérée en dimension $3$ et du changement de variable $t_i\leftrightarrow 1-t_i$ dans l'intégrale.) Ces relations de double mélange qui font intervenir des multi-indices potentiellement divergents s'appellent \emph{relations de double mélange étendu} et ont été introduites et étudiées par Ihara-Kaneko-Zagier \cite{iharakanekozagier}. On conjecture qu'elles contiennent toute l'information sur les relations linéaires entre valeurs zêta multiples:

\begin{conj}\label{conj:double shuffle}
Les relations de double mélange étendu engendrent toutes les relations $\QQ$-linéaires entre valeurs zêta multiples.
\end{conj}

Comme les relations de double mélange sont satisfaites par les \hbox{valeurs} zêta multiples motiviques, cette conjecture implique la conjecture des périodes (conjecture \ref{conj:periodes}). On ne connaît cependant pas de preuve élémentaire du fait qu'elle implique la formule de dimension de Zagier (conjecture \ref{conj:zagier dim}). La compatibilité entre la conjecture \ref{conj:double shuffle} et les conjectures du paragraphe \ref{sec:MZV} a été vérifiée expérimentalement en bas poids (voir par exemple \cite{minhetal} et \cite{mzvdatamine}).

\section{Périodes, cohomologie, motifs}\label{sec:motifs}

Ce paragraphe constitue un aparté introductif aux idées nécessaires pour comprendre la notion de valeur zêta multiple motivique. Nous renvoyons le lecteur à l'exposé de Javier Fresán dans ce volume pour plus de détails sur ces idées.

\subsection{Logarithme complexe, intégrales elliptiques, et périodes}

C'est un fait classique qu'il n'existe pas de fonction logarithme complexe, c'est-à-dire que l'équation différentielle $$f'(z)=\frac{1}{z}$$ n'a pas de solution dans l'espace des fonctions holomorphes \hbox{$f{:}\,\mathbb{C}^*\!\to\!\mathbb{C}$}. La preuve consiste à calculer une intégrale le long du chemin $\gamma:[0,1]\to \mathbb{C}^*$ parcourant le cercle unité dans le sens positif, c'est-à-dire donné par $\gamma(t)=\exp(2i\pi t)$:
\begin{equation}\label{eq:integral dz z}
\int_\gamma \frac{dz}{z} = \int_0^1 \frac{\gamma'(t)\, dt}{\gamma(t)} = \int_0^12i\pi \, dt = 2i\pi .
\end{equation}
Or le théorème de Stokes implique que l'intégrale d'une dérivée le long du chemin \emph{fermé} $\gamma$ doit s'annuler:\vspace*{-3pt}
\begin{align*}
\int_\gamma f'(z)\,dz = \int_0^1 f'(\gamma(t))\gamma'(t)\, dt&=\int_0^1 d(f(\gamma(t)))\\
&= f(\gamma(1))-f(\gamma(0))=0.
\end{align*}
Dit autrement, la formule\vspace*{-3pt}\enlargethispage{\baselineskip}%
$$\log(z)=\int_1^z\frac{dx}{x}$$
définit une fonction holomorphe \emph{multivaluée} de la variable complexe $z\in\mathbb{C}^*$, dont la valeur dépend du choix de chemin d'intégration entre $1$ et $z$ ; si on change le chemin en ajoutant un tour de $0$ dans le sens positif, on rajoute $2i\pi$ à la valeur de la fonction. 

La théorie des résidus de Cauchy explique que ce phénomène est la \emph{seule} obstruction au calcul de primitives de fonctions méromorphes avec singularité en $z=0$. Parallèlement, la théorie des résidus est un énoncé sur la topologie de $\mathbb{C}^*$: un lacet (chemin continu fermé) tracé dans $\mathbb{C}^*$ est contractile (c'est-à-dire peut être continûment contracté sur un point) si et seulement si l'intégrale de $\sfrac{dz}{z}$ sur ce chemin est nulle: le cercle unité représente donc la \emph{seule} obstruction à la déformation de lacets tracés dans $\mathbb{C}^*$. Remarquons que le nombre $2i\pi$ qui apparaît dans la comparaison entre ces deux notions d'obstruction via la formule \eqref{eq:integral dz z} est une période au sens de la définition \ref{defi:periodeKZ}. C'est aussi la période, au sens usuel, de la fonction exponentielle complexe, ce qui est une autre manière d'exprimer l'obstruction à l'existence d'un logarithme complexe.

Des considérations analogues sont au coeur de la théorie des intégrales elliptiques\footnote{Il s'agit ici d'intégrales elliptiques dites de \emph{première espèce}.}, de la forme
\begin{equation*}\label{eq:elliptic integral}
F(z)=\int_0^z \frac{dx}{\sqrt{(1-x^2)(1-k^2x^2)}},
\end{equation*}
où $k$ est une constante. Elles apparaissent au \siecle{18} dans le calcul de la période du pendule simple, ou encore dans le calcul du périmètre des ellipses (d'où elles tirent leur nom). On les interprète de manière moderne comme des intégrales sur la \emph{courbe elliptique}\footnote{Le lecteur avisé notera qu'il s'agit en fait d'une courbe elliptique privée du point à l'infini, qui apparaît en compactifiant $E$ dans le plan projectif $\mathbb{P}^2(\mathbb{C})$.}
$$E=\{(x,y)\in\mathbb{C}^2\sep y^2=(1-x^2)(1-k^2x^2)\}.$$
La fonction $F(z)$ est alors une fonction multivaluée sur $E$, car elle dépend du choix d'un chemin d'intégration. Comme une courbe ellip\-tique s'identifie topologiquement à la surface d'un tore, il y a à défor\-mation près deux contours d'intégration fermés indépendants $\gamma_1$ et~$\gamma_2$ sur $E$. Les intégrales de la forme différentielle $dx/y$ le long de $\gamma_1$ et~$\gamma_2$ s'appellent les \emph{périodes} de $E$ jouent le même rôle que $2i\pi$ dans la théorie des résidus. (Ce sont des périodes, au sens usuel, de la réciproque de $F(z)$, qu'on appelle une \emph{fonction elliptique} et qui joue le même rôle que la fonction exponentielle au paragraphe précédent. C'est dans ce contexte que s'est cristallisée l'utilisation du terme de \og période\fg au sens plus général de la définition \ref{defi:periodeKZ}.)

\subsection{Cohomologie et périodes}

Les discussions du paragraphe précédent s'interprètent dans le langage moderne comme une comparaison entre deux théories d'obstruction associées à une variété algébrique $X$ définie par des équations polynomiales à coefficients dans un sous-corps $K\subset \mathbb{C}$. Dans le cas de la théorie des résidus, $X=\mathbb{C}^*$ correspond aux solutions de l'équation $xy=1$ dans $\mathbb{C}^2$ et on peut prendre $K=\mathbb{Q}$. Dans le cas des intégrales elliptiques, $X=E$ et $K$ est un sous-corps de $\mathbb{C}$ qui contient le paramètre $k^2$.

\begin{enumerate}
\item D'une part on peut mesurer l'obstruction \emph{analytique} à calculer des primitives de formes différentielles sur $X$, mesurée par des groupes de \emph{cohomologie de de Rham} $H^n_{\mathrm{dR}}(X)$, où $n$ est le degré des formes différentielles. Une découverte importante de Grothendieck est que toutes ces obstructions viennent de formes différentielles \emph{algé\-briques}: les groupes de cohomologie de de Rham sont donc des espaces vectoriels définis sur le corps $K$ \cite{grothendieckderham}.
\item D'autre part on peut mesurer l'obstruction \emph{topologique} à défor\-mer des domaines tracées continûment sur $X$, mesurée par des groupes de \emph{cohomologie de Betti} (ou singulière) $H^n_{\mathrm{B}}(X)$, où $n$ est la dimension des domaines. Ce sont des espaces vectoriels définis sur $\mathbb{Q}$.
\end{enumerate}
La comparaison entre ces deux théories se fait par intégration des formes différentielles sur les domaines et donne lieu à l'isomorphisme suivant qui combine l'isomorphisme de de Rham \cite{derhamthese} et la comparaison, établie par Grothendieck, entre cohomologie de de Rham analytique et algébrique:\vspace*{-3pt}
\begin{equation}\label{eq:isom periodes}
\int: H^n_{\mathrm{dR}}(X)\otimes_K\mathbb{C} \stackrel{\simeq}{\longrightarrow} H^n_{\mathrm{B}}(X)\otimes_{\mathbb{Q}}\mathbb{C}.
\end{equation}
On l'appelle \emph{l'isomorphisme des périodes} et une matrice correspondante (relative à une $K$-base de $H^n_{\mathrm{dR}}(X)$ et une $\mathbb{Q}$-base de $H^n_{\mathrm{B}}(X)$) une \emph{matrice des périodes}. Ses coefficients sont certaines intégrales de formes différentielles algébriques sur des domaines tracés sur $X$ et sont des périodes au sens de la définition \ref{defi:periodeKZ} si $K$ est un corps de nombres (une extension finie de $\mathbb{Q}$). Réciproquement, on peut montrer que si on s'autorise à considérer des groupes de cohomologie \emph{relative} associés à des paires de variétés $(X,Y)$, toute période au sens de la définition \ref{defi:periodeKZ} apparaît dans une matrice des périodes pour une paire définie sur un corps de nombres \cite{hubermullerstachbook}.

Les exemples étudiés au paragraphe précédent concernent des inté\-grales le long de chemins, c'est-à-dire le cas $n=1$. La théorie des résidus nous apprend que les groupes de cohomologie $H^1_{\mathrm{dR/B}}(\mathbb{C}^*)$ sont de dimension $1$. La matrice des périodes contient la constante $\int_\gamma \sfrac{dz}{z} = 2i\pi$. Dans le cas d'une courbe elliptique, les groupes de cohomologie $H^1_{\mathrm{dR/B}}(E)$ sont de dimension~$2$. La matrice des périodes contient notamment les inté\-grales de $dx/y$ le long des chemins $\gamma_1$ et~$\gamma_2$\footnote{Elle contient aussi deux autres périodes appelées intégrales elliptiques de \emph{seconde espèce}, ou \emph{quasi-périodes} de $E$.}.

\subsection{La philosophie des motifs}

L'isomorphisme des périodes \eqref{eq:isom periodes} est un phénomène intriguant puisqu'il fait le lien entre deux notions d'obstruction \emph{a priori} très différentes. D'autres théories cohomologiques, mesurant d'autres obstructions encore, ont été définies au cours du \siecle{20} et s'avèrent être toutes comparables entre elles au sens où elles sont reliées par des isomorphismes canoniques sur le modèle de \eqref{eq:isom periodes}. La cohomologie étale, qui joue un rôle important dans la preuve des conjectures de Weil par Grothendieck et Deligne, en est un exemple important \cite{weilconjectures, grothendiecklefschetz, deligneweil1, deligneweil2}. Dans ce cas, la comparaison avec la cohomologie singulière permet de relier le comptage des solutions à un système d'équations polynomiales sur des corps finis à la topologie de la variété complexe correspondante.

Ces \og coïncidences\fg ont poussé Grothendieck à considérer la possibilité que les différentes cohomologies associées à une variété algébrique $X$ n'étaient que des \og projections\fg d'un objet plus primitif, le \emph{motif} de $X$. Contrairement à la cohomologie, ce motif n'est pas un objet algébrique comme un espace vectoriel, mais se manipule essentiellement de la même manière (par exemple, on peut parler de noyau ou d'image d'un morphisme entre motifs, prendre des produits cartésiens ou tensoriels, etc.)\footnote{Les motifs sont, selon l'idée de Grothendieck, des objets d'une catégorie \emph{tannakienne}.}. Les diffé\-ren\-tes théories de cohomologie sont alors vues comme des opérateurs\footnote{Ces opérateurs sont appelées des \emph{foncteurs de réalisation}.} qui permettent d'associer un espace vectoriel à un motif, par exemple
$$M\mto M_{\mathrm{dR}}\quad \text{et}\quad M\mto M_{\mathrm{B}}$$
pour la cohomologie de de Rham et la cohomologie de Betti. Une distinction fondamentale est qu'un morphisme entre motifs est toujours \og d'origine géométrique\fg (par exemple induit par un morphisme entre variétés algébriques, ou plus généralement par des cycles algébriques).

La mise en place de la théorie des motifs rêvée par Grothendieck a occupé les mathématiciens depuis maintenant un demi-siècle et les outils et techniques qu'ils ont mis en place à cette fin ont indéniablement bouleversé la géométrie algébrique et arithmétique. La théorie des valeurs zêta multiples motiviques dont il est question ici n'en est qu'une des nombreuses manifestations. Il y a plusieurs approches à la théorie des motifs et nous renvoyons le lecteur aux livres \cite{andrebook} et \cite{hubermullerstachbook} pour plus de détails sur les développements récents. Dans la suite de ce texte, nous nous affranchirons de ces difficultés en supposant fixée une théorie des motifs avec les propriétés attendues. Nous serons particulièrement intéressés par la famille de motifs qui s'obtiennent (en un sens précis) à partir du motif correspondant au groupe de cohomologie $H^1(\mathbb{C}^*)$, qui est noté $\mathbb{Q}(-1)$ et est appelé le \emph{motif de Lefschetz}. Ces motifs sont dits \emph{de Tate mixtes} et constituent une partie minuscule, mais déjà extrêmement riche, de la théorie des motifs.

\subsection{Périodes motiviques}\label{subsec:periodes motiviques}

On peut maintenant donner une définition (à faible degré d'imprécision) d'une période motivique:

\begin{defi}
Une \emph{période motivique} est un triplet $(M,\varphi,\omega)$ avec
\begin{itemize}
\item
$M$ un motif,
\item
$\omega\in M_{\mathrm{dR}}$ un élément de sa réalisation de de Rham et
\item
$\varphi\in M_{\mathrm{B}}^\vee$ une forme linéaire sur sa réalisation de Betti. 
\end{itemize}
On considère que deux périodes motiviques $(M,\varphi,\omega)$ et $(M',\varphi',\omega')$ sont égales s'il existe un morphisme de motifs $f:M\to M'$ tel que $f_{\mathrm{dR}}(\omega)=\omega'$ et $f_{\mathrm{B}}^\vee(\varphi')=\varphi$.

La \emph{période} associée à $(M,\varphi,\omega)$ est le nombre complexe
$$\mathrm{per}(M,\varphi,\omega)= \langle\varphi, \int \omega\rangle,$$
où $\int$ est l'isomorphisme de comparaison \eqref{eq:isom periodes}.
\end{defi}

Les périodes motiviques forment une algèbre\footnote{On utilise ici un symbole tilde pour éviter la confusion avec l'algèbre $\mathcal{P}^{\mathrm{mot}}$ définie plus haut dans ce texte, qui est une sous-algèbre de $\widetilde{\mathcal{P}}^{\mathrm{mot}}$.} $\widetilde{\mathcal{P}}^{\mathrm{mot}}$ et on a alors un morphisme d'algèbres, appelée le morphisme de période:
\begin{equation}\label{eq:morphisme per tilde}
\mathrm{per}:\widetilde{\mathcal{P}}^{\mathrm{mot}}\to\mathbb{C}.
\end{equation}
En pratique, $M$ est le motif d'une variété $X$ (ou plus généralement d'une paire de variétés), $\omega$ est la classe d'une forme différentielle algébrique sur $X$ et $\varphi$ est la classe d'un domaine tracé sur $X$. La période correspondante est alors tout simplement l'intégrale de $\omega$ sur $\varphi$. Il est donc aisé de construire des périodes motiviques correspondant à des périodes \emph{concrètes} données, comme $2i\pi$ ou les intégrales elliptiques étudiées plus haut. Par exemple, la version motivique de $2i\pi$ est
$$(2i\pi)^{\mathrm{mot}} = (\mathbb{Q}(-1),[\gamma],[dx/x]).$$
Une classe de morphismes entre motifs est donnée par les morphismes entre variétés algébriques, auquel cas l'égalité entre périodes motiviques prévue par la définition est une version abstraite de la formule de changement de variables du calcul intégral.

Une première intuition est donc qu'une période motivique est l'\emph{ingrédient géométrique} derrière une période. Une distinction fondamentale est que par définition, les relations entre périodes motiviques sont toutes \emph{de nature géométrique} car contrôlées par des morphismes entre motifs. La conjecture des périodes de Grothendieck affirme que c'est aussi le cas pour les périodes:

\begin{conj}[Conjecture des périodes de Grothendieck, cas général]\label{conj:periodes general}
Le morphisme de période \eqref{eq:morphisme per tilde} est injectif.
\end{conj}

C'est une conjecture très forte: elle affirme essentiellement que toute relation algébrique entre périodes peut être démontrée en n'utilisant que des outils algébriques. Dit autrement, il ne devrait faire aucune différence de travailler avec les périodes ou leurs versions motiviques. La conjecture \ref{conj:periodes general} est du même goût que la conjecture plus explicite suivante\footnote{La conjecture de Kontsevich-Zagier est essentiellement équivalente à la conjecture des périodes pour l'approche de la théorie des motifs due à Nori \cite{hubermullerstachbook}. On ne sait pas si cette approche donne la même théorie que d'autres approches, dont celle, due à Voevodsky, avec laquelle nous travaillons implicitement dans ce texte \cite{voevodskytriangulated}.}:

\begin{conj}[Conjecture de Kontsevich-Zagier, \cite{kontsevichzagier}]
Toutes les relations linéaires entre périodes sont des conséquences~de:
\begin{enumerate}
\item la bilinéarité de l'intégration ;
\item la formule de changement de variable ;
\item la formule de Stokes.
\end{enumerate}
\end{conj}

Remarquons que les relations de double mélange entre valeurs zêta multiples ne sont pas directement de nature géométrique puisque le produit de quasi-mélange résulte de la manipulation de séries. On peut cependant les démontrer en ne manipulant que des intégrales \cite{souderes}, ce qui montre qu'elles vérifient les prédictions des deux conjectures ci-dessus.

\section{Intégrales itérées et groupe fondamental de \texorpdfstring{$\mathbb{P}^1\backslash\{0,1,\infty\}$}{P1}}\label{sec:pi1mot}

On a vu à la proposition \ref{prop:MZVs it int} que les valeurs zêta multiples ont des représen\-ta\-tions intégrales. Nous développons maintenant cette idée en faisant intervenir le groupe fondamental de l'espace $\mathbb{P}^1\backslash\{0,1,\infty\}$. C'est ce point de vue géométrique qui permet de définir et manipuler les valeurs zêta multiples motiviques.

\subsection{Intégrales itérées et polylogarithmes}

On étudie l'espace $X=\mathbb{P}^1(\mathbb{C})\setminus\{0,1,\infty\} = \mathbb{C}\setminus\{0,1\}$. On rappelle les deux formes différentielles élémentaires de degré $1$ sur $X$:
$$\omega_0 = \frac{dx}{x} \quad \text{et} \quad \omega_1 = \frac{dx}{1-x}\cdot$$

Pour deux points-base $a,b\in X$, considérons des chemins continus $\gamma:[0,1]\to X$ tels que $\gamma(0)=a$ et $\gamma(1)=b$. L'ensemble de tels chemins modulo la relation d'homotopie (qui identifie deux chemins qui peuvent être déformés continûment l'un en l'autre) est noté $\pi_1(X,a,b)$. La composition des chemins induit une application
\begin{equation}\label{eq:composition chemins pi1}
\pi_1(X,a,b)\times \pi_1(X,b,c)\to \pi_1(X,a,c).
\end{equation}
Cela fait notamment de $\pi_1(X,a,a)=:\pi_1(X,a)$ un groupe, appelé le \emph{groupe fondamental} de $X$ basé en $a$. Plus généralement, la collection de tous les ensembles $\pi_1(X,a,b)$ munie des opérations de composition~\eqref{eq:composition chemins pi1} forme le \emph{groupoïde fondamental} de $X$. On montre facilement que le groupe fondamental ne dépend pas du point-base et est isomorphe à un groupe libre sur deux générateurs $\gamma_0$ et $\gamma_1$, qui sont les (classes des) lacets autour de $0$ et $1$ respectivement. Par exemple, pour $a=\sfrac{1}{2}$ on peut choisir $\gamma_0(t)=\frac{1}{2}\exp(2i\pi t)$ et $\gamma_1(t)=1-\frac{1}{2}\exp(2i\pi t)$.

Pour un chemin $\gamma:[0,1]\to X$ tel que $\gamma(0)=a$ et $\gamma(1)=b$, et pour tout mot $\omega_{i_1}\cdots\omega_{i_n}$ avec $i_k\in\{0,1\}$, on peut considérer l'\emph{intégrale itérée}
\begin{equation}\label{eq:def it int general}
\int_\gamma\omega_{i_1}\cdots\omega_{i_n} = \int_{0\leq t_1\leq \cdots \leq t_n\leq 1} \omega_{i_1}(\gamma(t_1)) \cdots \omega_{i_n}(\gamma(t_n)).
\end{equation}
Pour $n=1$ ce sont simplement des intégrales de $1$-formes le long de chemins, et on obtient des valeurs spéciales de la fonction logarithme:
$$\int_\gamma\omega_0 = \log(b/a).$$
Un exemple d'intégrale itérée de longueur $n=2$ est:
$$\int_\gamma\omega_0\omega_1 = \int_\gamma \frac{dx}{1-x}\, \log(x/a).$$
Après changement de variable, elle peut s'exprimer en fonction d'une valeur spéciale de la fonction \emph{dilogarithme}:
$$\mathrm{Li}_2(z)= \sum_{k=1}^\infty \frac{z^k}{k^2} = - \int_0^z \frac{dx}{x}\,\log(1-x).$$
Cette fonction est bien définie sur le disque unité $\{|z|<1\}$ et s'étend en une fonction \emph{multivaluée} sur $X$. Les intégrales itérées de longueur quelconque contiennent des valeurs spéciales des fonctions \emph{polylogarithmes}
$$\mathrm{Li}_n(z)= \sum_{k=1}^\infty \frac{z^k}{k^n} = \int_0^z\frac{dx}{x}\,\mathrm{Li}_{n-1}(z),$$
et plus généralement des \emph{polylogarithmes multiples}:
\begin{equation}\label{eq:polylog multiple}
\mathrm{Li}_{n_1,\ldots,n_r}(z_1,\ldots,z_r) = \sum_{1\leq k_1<\cdots <k_r} \frac{z_1^{k_1}\cdots z_r^{k_r}}{k_1^{n_1}\cdots k_r^{n_r}} \cdot
\end{equation}

\subsection{L'isomorphisme de Chen}

La théorie des intégrales itérées de Chen suit la même philosophie que les travaux de de Rham \cite{derhamthese} sur l'homologie singulière ayant mené à l'isomorphisme des périodes \eqref{eq:isom periodes}: elle vise à comprendre le groupe fondamental de $X$ en termes de formes différentielles via les inté\-grales itérées. Le groupe fondamental étant un objet profondément \og non abélien\fg, la théorie de Chen ne concerne qu'une approximation linéaire du groupe fondamental, appelée sa \emph{complétion pro-unipotente} ou \emph{complétion de Malcev}. Pour la définir, notons $G=\nobreak\pi_1(X,a)$ et soit $\mathbb{C}[G]$ son algèbre de groupe, dont une base est donnée par l'ensemble $G$ et la multiplication est donnée par l'extension linéaire de la loi de groupe. On a un idéal $I\subset\nobreak \mathbb{C}[G]$, l'\emph{idéal d'augmentation}, engendré par les différences de deux éléments de $G$. On~considère ses puissances $I^{n+1}$, qui sont des idéaux de plus en plus petits dans $\mathbb{C}[G]$. Quitte à choisir un chemin de $a$ vers $b$, on a une bijection $\pi_1(X,a)\simeq \pi_1(X,a,b)$ et on peut voir les idéaux $I^{n+1}$ comme des sous-espaces vectoriels de l'espace vectoriel $\mathbb{C}[\pi_1(X,a,b)]$\footnote{Ces sous-espaces vectoriels ne dépendent pas du choix du chemin de $a$ vers $b$.}. Les~quotients $\mathbb{C}[\pi_1(X,a,b)]/I^{n+1}$ sont des approximations linéaires de plus en plus fidèles de $\pi_1(X,a,b)$.

Notons $\mathbb{C}\langle\omega_0,\omega_1\rangle$ l'espace des polynômes non commutatifs en les symboles $\omega_0$, $\omega_1$, et $\mathbb{C}\langle \omega_0,\omega_1\rangle_{\leq n}$ le sous-espace des polynômes de degré $\leq n$, engendré par les mots (monômes) de longueur $\leq n$. L'intégrale itérée \eqref{eq:def it int general} est invariante par homotopie et induit donc, pour chacun de ces mots, une forme linéaire sur l'espace $\mathbb{C}[\pi_1(X,a,b)]$, dont on peut montrer qu'elle s'annule sur le sous-espace $I^{n+1}$.

\begin{thm}[Chen \cite{chen}]
Les intégrales itérées induisent un isomorphisme de $\mathbb{C}$-espaces vectoriels:
\begin{equation}\label{eq:isom chen}
\int: \mathbb{C}\langle\omega_0,\omega_1\rangle_{\leq n} \stackrel{\simeq}{\longrightarrow} \left(\mathbb{C}[\pi_1(X,a,b)]/I^{n+1}\right)^\vee.
\end{equation}
\end{thm}

Dans la version la plus générale du théorème, $X$ peut être remplacé par n'importe quelle variété différentielle. Le terme de gauche de l'isomorphisme de Chen doit être alors remplacé par un objet qui sélectionne les intégrales itérées qui sont invariantes par homotopie (cette condition est automatiquement satisfaite dans le cas qu'on étudie ici).

\subsection{Groupe fondamental motivique et valeurs zêta multiples motiviques}\label{subsec:groupe fondamental motivique}

L'isomorphisme de Chen \eqref{eq:isom chen} est d'apparence analogue à l'isomorphisme des périodes \eqref{eq:isom periodes} puisqu'il compare un espace vectoriel de formes différentielles algébriques à un invariant topologique de $X$. Cette analogie peut être rendue précise par un théorème attribué par Goncharov à Beilinson \cite[Th.\,4.1]{goncharovmultiplepolylogs}, qui affirme que l'isomorphisme de Chen est un cas particulier de l'isomorphisme des périodes pour un certain groupe de cohomologie (relative). On peut alors considérer le motif qui correspond à ce groupe de cohomologie puis faire tendre $n$ vers $+\infty$ ; l'objet obtenu est abusivement appelé le \emph{groupe fondamental motivique} de $X$ avec points base $a$,~$b$. Le slogan est donc: les intégrales itérées sont les périodes du groupe fondamental motivique. La théorie des groupes fondamentaux motiviques a été développée par Deligne et Goncharov \cite{delignedroiteprojective, goncharovdihedral,goncharovgaloissym,delignegoncharov} et fait écho à des travaux plus anciens de Hain et Hain-Zucker sur les structures de Hodge sur le groupe fondamental \cite{hainhomotopy, hainderham, haingeometry, hainzucker}.

Afin d'étudier les valeurs zêta multiples, il faut d'après la proposition \ref{prop:MZVs it int} considérer des intégrales itérées \eqref{eq:def it int general} pour le choix de chemin $\gamma=[0,1]$. Or les extrémités de ce chemin ne sont pas contenues dans $X=\mathbb{P}^1(\mathbb{C})\setminus\{0,1,\infty\}$ et la théorie que nous venons d'esquisser ne s'applique pas. On a alors besoin d'une variante où les points-base $a$ et $b$ ne sont plus des points de $X$ mais des \emph{points-base tangentiels à l'infini}, c'est-à-dire des directions tangentes en un point $0$, $1$ ou $\infty$ qui est \og à l'infini\fg du point de vue de l'espace $X$. Les intégrales itérées peuvent alors être divergentes et pour donner un sens à l'isomorphisme de Chen \eqref{eq:isom chen} il faut les \emph{régulariser}: on retrouve alors la régularisation de mélange déjà évoquée au paragraphe \ref{subsec:shuffle}. 

Une fois ces constructions réalisées, on est alors muni d'un motif~$M$ satisfaisant à:
\begin{itemize}
\item la réalisation de de Rham de $M$ est l'espace $M_{\mathrm{dR}}=\mathbb{C}\langle\omega_0,\omega_1\rangle$. On a pour tout multi-indice $\underline{n}=(n_1,\ldots,n_r)$, convergent ou non, un élément
$$\Omega_{\underline{n}}=\omega_1\omega_0^{n_1-1}\cdots \omega_1\omega_0^{n_r-1} \in M_{\mathrm{dR}}\,;$$
\item
la réalisation de Betti $M_{\mathrm{B}}$ est l'espace des fonctions sur le complété pro-unipotent du groupe fondamental de $X$ avec les points-base tangentiels adéquats en $0$ et en $1$. Le segment unité définit un élément
$$[0,1] \in M_{\mathrm{B}}^\vee.$$
\end{itemize}
En suivant la recette du paragraphe \ref{subsec:periodes motiviques}, on peut alors donner la définition des valeurs zêta multiples motiviques:

\begin{defi}
On définit la \emph{valeur zêta multiple motivique}
$$\zeta^{\mathrm{mot}}(n_1,\ldots,n_r) = (M,[0,1], \Omega_{\underline{n}}) \in\widetilde{\mathcal{P}}^{\mathrm{mot}}.$$
\end{defi}

La période associée à $\zeta^{\mathrm{mot}}(n_1,\ldots,n_r)$ est par définition l'intégrale itérée de $\Omega_{\underline{n}}$ entre $0$ et $1$, c'est-à-dire le nombre $\zeta(n_1,\ldots,n_r)$ par la proposition \ref{prop:MZVs it int}. Dans le cas où $\underline{n}$ est un multi-indice divergent, ce nombre s'interprète via la régularisation de mélange.

L'image du segment unité $[0,1]\in M_{\mathrm{B}}^\vee$ par l'isomorphisme de comparaison de $M$ s'appelle l'\emph{associateur de Drinfel'd}. C'est une série formelle en deux variables dont les coefficients sont toutes les valeurs zêta multiples, qui est un des points de départ de la théorie (pro-unipotente) de Grothendieck-Teichmüller \cite{drinfeldquasihopf}.

\section{Applications des valeurs zêta multiples motiviques}\label{sec:applis pi1mot}

Nous expliquons maintenant comment les idées motiviques permettent de prouver le théorème \ref{thm:goncharov terasoma} de Goncharov et Terasoma et le théorème \ref{thm:brown} de Brown.

\subsection{Des nombres algébriques aux périodes}

Les nombres algébriques sont des périodes particulières au sens de la définition \ref{defi:periodeKZ}, par exemple on peut écrire $\sqrt{2}=\int_{x^2\leq 2}\frac{1}{2} dx$. D'un point de vue plus cohomologique, une racine d'un polynôme irréductible $f(x)\in\mathbb{Q}[x]$ apparaît dans la matrice des périodes du groupe de cohomologie $H^0(X)$ avec $X=\{x\in\mathbb{C}\;|\; f(x)=0\}$. La cohomologie de de Rham est simplement l'espace des fonctions $\mathbb{Q}[x]/(f(x))$ et la cohomologie de Betti a pour base les points de $X$. Une \hbox{matrice} des périodes est donc la matrice de Vandermonde associée aux \hbox{racines} de~$f$. En~résumé, les nombres algébriques sont les périodes des variétés algébriques sur~$\mathbb{Q}$ qui sont de dimension $0$. 

La théorie de Galois étudie les \emph{symétries} des équations polynomiales en une variable. Au corps de nombres $K=\mathbb{Q}[x]/(f(x))$ est associé son \emph{groupe de Galois} $G$ qui est par définition le groupe des automorphismes du corps~$K$. C'est un groupe fini qui se plonge dans le groupe des permutations de l'ensemble des racines complexes de~$f$. Un élément de $G$ est uniquement déterminé par son action sur le générateur $x$ et le cardinal de $G$ est au plus le degré de l'extension $[K:\mathbb{Q}]=\deg(f)$. On dit que l'extension est \emph{galoisienne} si on a l'égalité $|G|=[K:\mathbb{Q}]$. Dans ce cas, la \emph{correspondance de Galois} établit une bijection entre les sous-corps de $K$ et les sous-groupes de $G$. 

Une partie de cette théorie se généralise (conjecturalement) aux périodes de la manière suivante \cite{andregalois}. Pour un motif quelconque~$M$ on peut définir un groupe $G_M$, dit \emph{groupe de Galois motivique}, dont les éléments sont les automorphismes linéaires de l'espace vectoriel $M_{\mathrm{dR}}$ qui \og respectent toutes les constructions algébriques entre motifs\fg\footnote{On a besoin du \emph{formalisme tannakien} pour donner un vrai sens mathématique à l'expression entre guillemets.}. On montre que c'est un groupe algébrique, c'est-à-dire un groupe de matrices carrées défini par des équations algébriques. Il~agit sur les périodes motiviques de $M$ par la formule
$$g\cdot (M,\varphi,\omega)=(M,\varphi,g\cdot \omega),$$
dit autrement il agit \og par multiplication à droite de la matrice des périodes\fg. Si on croit à la conjecture des périodes de Grothendieck (conjecture~\ref{conj:periodes general}), alors cette action descend à une action sur les \hbox{périodes} de $M$. Sans cette conjecture, on n'a accès qu'à une théorie de Galois des périodes \emph{motiviques}. Notons qu'une reformulation de la conjecture des périodes est que la dimension de $G_M$ est égale au degré de transcendance de l'algèbre engendrée par les périodes de $M$. 

Toute l'intuition liée aux nombres algébriques ne se transfère pas au cas des périodes, mais notons tout de même deux principes généraux que nous revisiterons par la suite:

\begin{itemize}
\item \emph{Principe de finitude}. Soit $K$ un sous-corps de $\mathbb{C}$ engendré par un nombre fini de nombres algébriques. Alors $K$ est de dimension finie sur $\mathbb{Q}$.
\item \emph{Principe de reconnaissance}. Supposons que l'extension $K/\mathbb{Q}$ est galoisienne et soit $G$ son groupe de Galois. Si $x\in K$ est un point fixe de l'action de $G$, alors $x\in\mathbb{Q}$.
\end{itemize}
Le premier principe est essentiellement tautologique puisqu'un nom\-bre est algébrique exactement quand ses puissances engendrent un espace vectoriel de dimension finie sur $\mathbb{Q}$. Une application triviale mais parlante du second principe est la suivante: considérons le corps $\mathbb{Q}(\sqrt{2})$ constitué des nombres $a+b\sqrt{2}$ avec $a,b\in\mathbb{Q}$. C'est une extension galoisienne de $\mathbb{Q}$ dont le groupe de Galois est $\mathbb{Z}/2\mathbb{Z}$, engendré par la \og conjugaison\fg $\tau:a+b\sqrt{2}\mto a-b\sqrt{2}$. Comme~$\tau$ est un morphisme de corps, il respecte somme et produit et agit trivialement sur le nombre $x_n=(1+\sqrt{2})^n+(1-\sqrt{2})^n$. On en déduit, sans faire aucun calcul, que $x_n$ est un nombre rationnel.

\subsection{Le principe de finitude et la borne de Goncharov-Terasoma}

L'application du principe de finitude au cas des valeurs zêta multiples permet de prouver le théorème \ref{thm:goncharov terasoma} de Goncharov et Terasoma. Soit $M$ le motif défini au paragraphe \ref{subsec:groupe fondamental motivique} dont les périodes contiennent les valeurs zêta multiples. Il fait partie d'une famille de motifs qui s'appellent \emph{motifs de Tate mixtes sur~$\mathbb{Z}$} et qui sont très particuliers pour deux raisons:
\begin{enumerate}
\item Ils peuvent tous s'obtenir (en un sens précis) à partir du motif de Lefschetz $\mathbb{Q}(-1)$. Cela exclut les motifs des courbes elliptiques, par exemple.
\item Ils sont \emph{à bonne réduction} en chaque nombre premier $p$, c'est-à-dire que la réduction modulo $p$ ne change pas la géométrie sous-jacente au motif. Cela exclut les groupes fondamentaux motiviques de $\mathbb{P}^1\setminus\{0,1,\infty\}$ à points-base finis ; par exemple, le point-base $a=2$ se comporte mal vis-à-vis de la réduction modulo $p=2$.
\end{enumerate}

Ces contraintes fortes impliquent une borne sur la taille de cette famille. Concrètement, les périodes motiviques de ces motifs vivent dans une sous-algèbre $\mathcal{P}^{\mathrm{mot}}$ de l'algèbre $\widetilde{\mathcal{P}}^{\mathrm{mot}}$ de toutes les périodes motiviques, qui est \emph{petite} au sens où elle est de dimension finie en tout poids. On peut quantifier sa taille et même sa structure abstraite grâce à des résultats profonds, notamment le calcul par Borel de la $K$-théorie rationnelle de l'anneau des entiers \cite{borel}. On obtient l'isomorphisme
$$\mathcal{P}^{\mathrm{mot}} \simeq \mathcal{F}$$
promis au paragraphe \ref{subsec:MZV mot first}, qui donne les dimensions de $\mathcal{P}^{\mathrm{mot}}$ en tout poids, et le théorème \ref{thm:goncharov terasoma} après application du morphisme de période\footnote{\emph{Stricto sensu}, ce qu'on note $\mathcal{P}^{\mathrm{mot}}$ est l'algèbre des périodes motiviques \emph{effectives} et \emph{réelles} de la catégorie des motifs de Tate mixtes sur $\mathbb{Z}$.}.

\subsection{Le principe de reconnaissance et le théorème de Brown}

L'action du groupe de Galois motivique $G_M$ sur les valeurs zêta multiples motiviques est étonnamment simple à calculer explicitement par une formule combinatoire due à Goncharov \cite{goncharovgaloissym} et raffinée par Brown \cite{brownMTMZ}. Les valeurs zêta (motiviques) simples jouent un rôle spécial vis-à-vis de cette action au sens où elles ont un espace de conjugués minimal. Plus précisément on a la formule, pour $n\geq 2$:
$$g\cdot \zeta^{\mathrm{mot}}(n) = a_g^{(n)}\zeta^{\mathrm{mot}}(n) + b_g^{(n)}$$
où $a_g^{(n)}$ et $b_g^{(n)}$ sont des fonctions sur $G_M$ et $b_g^{(n)}=0$ pour $n$ pair. (Cette annulation est une manifestation de la solution d'Euler au problème de Bâle qui dit que les valeurs zêta paires sont des multiples rationnels de puissances de $\pi$.) Une remarque importante est que cette formule \emph{caractérise} les valeurs zêta (motiviques) simples, ce qui est une variante du principe de reconnaissance:

\begin{prop}\label{prop:reconnaissance zeta simples}
Soit $\xi\in\mathcal{P}^{\mathrm{mot}}$ une période motivique de poids $n\geq 2$ telle que $G\cdot \xi\subset \mathbb{Q}\xi\oplus\mathbb{Q}$. Alors $\xi$ est un multiple rationnel de $\zeta^{\mathrm{mot}}(n)$.
\end{prop}

Une découverte de Brown est que cette remarque permet de montrer l'existence de relations entre valeurs zêta multiples par voie galoisienne \cite{browndecomposition}. Par exemple, partant de la formule
$$g\cdot \zeta^{\mathrm{mot}}(2,3) = a_g^{(5)}\,\zeta^{\mathrm{mot}}(2,3) + 3\,a_g^{(2)}b_g^{(3)}\, \zeta^{\mathrm{mot}}(2) + c_g^{(5)}$$
pour la théorie de Galois de $\zeta^{\mathrm{mot}}(2,3)$, on déduit que la combinaison linéaire
$$\zeta^{\mathrm{mot}}(2,3)-3\,\zeta^{\mathrm{mot}}(2)\zeta^{\mathrm{mot}}(3)$$
satisfait aux hypothèses de la proposition \ref{prop:reconnaissance zeta simples} et est donc un multiple rationnel de $\zeta^{\mathrm{mot}}(5)$:
$$\zeta^{\mathrm{mot}}(2,3)-3\,\zeta^{\mathrm{mot}}(2)\zeta^{\mathrm{mot}}(3) = \lambda\, \zeta^{\mathrm{mot}}(5) \quad (\lambda\in\mathbb{Q}).$$
Le coefficient de proportionnalité exact ne peut pas être déterminé par cette méthode mais peut être calculé empiriquement en appliquant le morphisme de période:
$$\lambda = \frac{\zeta(2,3)-3\,\zeta(2)\zeta(3)}{\zeta(5)},$$
qui peut se calculer avec une précision arbitraire\footnote{On trouve alors $\lambda\simeq -11/2$. Dans ce cas simple, on peut évidemment montrer la relation en question grâce aux relations de double mélange.}. Cette stratégie est au centre de la preuve des théorèmes \ref{thm:brown} et \ref{thm:brown bis} de Brown.

\section{Les valeurs zêta multiples comme amplitudes en~physique}

\subsection{Amplitudes et intégrales de Feynman}

À la fin des années 1940, Feynman révolutionne la physique en introduisant une méthode pour calculer l'amplitude de probabilité des interactions entre particules élémentaires, qui a très rapidement fait ses preuves et est encore utilisée aujourd'hui. Ces amplitudes se calculent comme des séries entières dont le paramètre est appelé \emph{constante de couplage}, et dont les coefficients sont des intégrales (de~Feynman) indexées par des graphes (de Feynman). Chaque graphe témoi\-gne d'un scénario de collision, où les arêtes représentent des particules et les sommets des interactions. La précision du résul\-tat augmente au fur et à mesure que l'on calcule des intégrales de \hbox{Feynman} associées à des graphes de plus en plus complexes. (La~complexité est mesurée par le nombre de cycles indépendants dans le graphe.) On peut alors vérifier que ces calculs décrivent fidèlement la réalité en comparant leurs prédictions aux mesures issues des accélérateurs de particules, par exemple le LHC (\emph{Large Hadron Collider} ou Grand collisionneur de hadrons, gigantesque accélérateur du CERN situé sous la frontière franco-suisse). Les intégrales de Feynman sont notoirement difficiles à calculer, même dans le cas d'interactions simples, mais dans les cas abordables, décrivent la réalité de manière suprenamment bonne.

Un de ces cas est la théorie $\phi^4$ non massive, qui correspond à des graphes (au sens usuel) où chaque sommet est adjacent à au plus $4$ arêtes. Pour un tel graphe $G$, supposé connexe, soient $e_1,\ldots,e_{N_G}$ ses arêtes et définissons un polynôme
$$\Psi_G=\sum_{T}\prod_{e\notin T}x_e,$$
où la somme est prise sur les arbres couvrants $T$ de $G$ (ensembles d'arêtes maximaux sans cycles), et où $x_e$ est une variable formelle (appelée \emph{paramètre de Schwinger}) associée à chaque arête $e$. Ce polynôme est d'abord apparu dans les travaux de Kirchhoff sur les circuits électriques. On définit alors une intégrale
\begin{equation}\label{eq:integrale graphe}
I_G=\int_\sigma \frac{\Omega_G}{\Psi_G^2},
\end{equation}
où $\sigma$ est le domaine $\{x_e\geq 0 \;(\forall e)\}$ de l'espace projectif $\mathbb{P}^{N_G-1}$ et
$$\Omega_G = \sum_{i=1}^{N_G}(-1)^{i}x_i \, dx_1\wedge\cdots \wedge \widehat{dx_i}\wedge \cdots\wedge dx_{N_G}$$
est la forme volume canonique. Il est facile de caractériser les graphes, appelés log-divergents primitifs, pour lesquels l'intégrale ci-dessus converge. Dans ce cas sa valeur est une période au sens de la définition \ref{defi:periodeKZ} et est une contribution à une amplitude de la théorie $\phi^4$. On ne sait essentiellement pas calculer les intégrales \eqref{eq:integrale graphe} pour des graphes avec plus de 12 cycles indépendants.

L'étude géométrique des amplitudes \eqref{eq:integrale graphe} a été initiée par Bloch-Esnault-Kreimer \cite{blochesnaultkreimer}, qui associent à $G$ un groupe de cohomologie relative dont la matrice des périodes contient $I_G$. La compréhension de ce groupe de cohomologie passe par l'étude de l'\emph{hypersurface de graphe} $X_G=\{\Psi_G=0\}$ et relève notamment de la théorie des singularités.

\subsection{Apparition des valeurs zêta multiples}

Les intégrales \eqref{eq:integrale graphe} ont été calculées de manière systématique dans les années 1990, notamment par Broadhurst et ses collaborateurs, voir par exemple Broadhurst-Kreimer \cite{broadhurstkreimerknots}. Pour certains graphes, les amplitudes s'expriment en termes de valeurs zêta simples:

\begin{figure}[htb]
\begin{center}
\includegraphics[scale=.5]{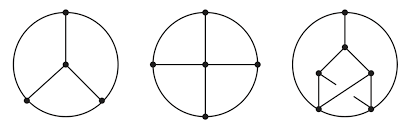}
\end{center}
\hspace*{.2cm}$6\,\zeta(3)$\hspace{1.4cm}$20\,\zeta(5)$\hspace{1.33cm}$36\,\zeta(3)^2$\vspace*{-5pt}
\end{figure}
On trouve aussi des valeurs zêta multiples qui (conjecturalement) ne sont pas des polynômes en les valeurs zêta simples:\enlargethispage{\baselineskip}%

\begin{figure}[htb]
\begin{center}
\includegraphics[scale=.35]{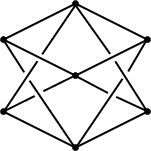}
\end{center}
\centerline{$\frac{27}{5}\zeta(5,3)+\frac{45}{4}\zeta(5)\zeta(3)-\frac{261}{20}\zeta(8)$}
\end{figure}
Des variantes \emph{cyclotomiques} des valeurs zêta multiples apparaissent aussi ; ce sont les évaluations aux racines de l'unité des polylogarithmes multiples~\eqref{eq:polylog multiple}. On ne sait pas quelles classes de périodes apparaissent comme des amplitudes de la forme \eqref{eq:integrale graphe} mais il est aujourd'hui communément admis que contrairement à ce que montrent les calculs pour des petits graphes, la géométrie sous-jacente à ces amplitudes peut être \og aussi compliquée que possible\fg \cite{belkalebrosnan, brownschnetzK3}. On s'attend donc à rencontrer des périodes de plus en plus riches à mesure que les techniques de calcul se développeront.

\subsection{Vers une théorie de Galois cosmique}

Il est notable que les amplitudes de la théorie $\phi^4$ qui sont des combinaisons linéaires de valeurs zêta multiples engendrent un petit sous-espace $\mathcal{Z}_{\phi^4}$ de l'algèbre $\mathcal{Z}$. Par exemple, on est amené à penser que ni $\zeta(2)$ ni aucun produit $\zeta(2)\zeta(2n+1)$ n'appartient à ce sous-espace. Il y a donc une tension entre la richesse de ces amplitudes que nous venons d'évoquer et leur \emph{rareté}. Ce phénomène reste encore mystérieux mais pourrait avoir une origine galoisienne. En effet, les calculs de Panzer-Schnetz suggèrent que la version motivique $\mathcal{Z}_{\phi^4}^{\mathrm{mot}}$ est stable par l'action du groupe de Galois motivique des valeurs zêta multiples \cite{panzerschnetz}. Plus ambitieusement, Panzer et Schnetz proposent la conjecture suivante, qui va bien au-delà du cadre des valeurs zêta multiples:

\begin{conj}[Conjecture de coaction]\label{conj:coaction}
L'espace engendré par toutes les amplitudes (motiviques) de la théorie $\phi^4$ est stable par l'action du groupe de Galois motivique.
\end{conj}

Cette conjecture affirme que les contraintes sur les périodes qui apparaissent en théorie $\phi^4$ se propagent par l'action du groupe de Galois motivique\footnote{Reste encore à prouver que des contraintes existent, ce qui peut se faire par un nombre fini de calculs d'après le \emph{principe des petits graphes} de Brown \cite{brownfeynmancosmic}, qui affirme que des périodes de poids \og petit\fg doivent apparaître comme des amplitudes de graphes \og petits\fg.}. Elle suggère donc des symétries cachées dans les calculs d'amplitudes en physique des particules et donne du poids à l'idée, popularisée par Cartier, d'une \og théorie de Galois cosmique\fg de certaines constantes physiques \cite{cartiermad}.

\subsubsection*{Pour aller plus loin}

L'étude des valeurs zêta multiples est à l'intersection de nombreux domaines des mathématiques et de la physique. Nous ne pouvions donc aborder que certains des aspects de la théorie et laisser de côté des sujets aussi passionnants que les liens avec les associateurs et la théorie de Grothendieck-Teichmüller, ou encore les développements récents liés aux phénomène modulaires dans les valeurs zêta multiples. Le lecteur curieux pourra trouver des pistes de réflexion et une bibliographie plus étoffée dans le livre de référence \cite{burgosfresan}.

\bibliographystyle{amsalpha}
\bibliography{Expose-X-UPS-final}

\providecommand{\eprint}[1]{\href{http://arxiv.org/abs/#1}{\texttt{arXiv\string:\allowbreak#1}}}
\providecommand{\bysame}{\leavevmode\hbox to3em{\hrulefill}\thinspace}
\providecommand{\MR}{\relax\ifhmode\unskip\space\fi MR }
\providecommand{\MRhref}[2]{%
  \href{http://www.ams.org/mathscinet-getitem?mr=#1}{#2}
}
\providecommand{\href}[2]{#2}
\begin{thebibliography}{HJPO99}

\bibitem[And04]{andrebook}
Y.~Andr{\'e}, \emph{Une introduction aux motifs (motifs purs, motifs mixtes,
  \hbox{p{\'e}riodes})}, Panoramas \& Synth{\`e}ses, vol.~17, Soci{\'e}t{\'e}
  Math{\'e}matique de France, Paris, 2004.

\bibitem[And09]{andregalois}
\bysame, \emph{Galois theory, motives and transcendental numbers},
  Renormalization and {G}alois theories, IRMA Lect. Math. Theor. Phys.,
  vol.~15, Eur. Math. Soc., Z{\"u}rich, 2009, pp.~165--177.

\bibitem[And17]{andrebourbaki}
\bysame, \emph{Groupes de {G}alois motiviques et p{\'e}riodes}, {S{\'e}minaire
  Bourbaki. Vol. 2015/16}, Ast{\'e}risque, vol. 390, Soci{\'e}t{\'e}
  Math{\'e}matique de France, Paris, 2017, Exp.\,No.\,1104, pp.~1--26.

\bibitem[Ap{\'e}79]{apery}
R.~Ap{\'e}ry, \emph{Irrationalit{\'e} de {$\zeta 2$} et {$\zeta 3$}}, Journées
  arithmétiques de Luminy, Ast{\'e}risque, vol.~61, Soci{\'e}t{\'e}
  Math{\'e}matique de France, Paris, 1979, pp.~11--13.

\bibitem[BB03]{belkalebrosnan}
P.~Belkale and P.~Brosnan, \emph{Matroids motives, and a conjecture of
  {K}ontsevich}, Duke Math.~J. \textbf{116} (2003), no.~1, 147--188.

\bibitem[BBV10]{mzvdatamine}
J.~Bl{\"u}mlein, D.~J. Broadhurst, and J.~A.~M. Vermaseren, \emph{The multiple
  zeta value data mine}, Comput. Phys. Comm. \textbf{181} (2010), no.~3,
  582--625.

\bibitem[BEK06]{blochesnaultkreimer}
S.~Bloch, H.~Esnault, and D.~Kreimer, \emph{On motives associated to graph
  polynomials}, Comm. Math. Phys. \textbf{267} (2006), no.~1, 181--225.

\bibitem[Beu79]{beukersapery}
F.~Beukers, \emph{A note on the irrationality of {$\zeta (2)$} and {$\zeta
  (3)$}}, Bull. London Math. Soc. \textbf{11} (1979), no.~3, 268--272.

\bibitem[BF89]{baileyferguson}
D.~H. Bailey and H.~R.~P. Ferguson, \emph{Numerical results on relations
  between fundamental constants using a new algorithm}, Math. Comput.
  \textbf{53} (1989), no.~188, 649--656.

\bibitem[BF91]{pslq}
\bysame, \emph{A polynomial time, numerically stable integer relation
  algorithm}, Report SRC-TR-92-066, Supercomputing Research Center, 1991.

\bibitem[BGF17]{burgosfresan}
J.~I. Burgos~Gil and J.~Fres\'{a}n, \emph{Multiple zeta values: from numbers to
  motives}, Clay Mathematics Proceedings, American Mathematical Society,
  Providence, RI, 2017.

\bibitem[BK97]{broadhurstkreimerknots}
D.~J. Broadhurst and D.~Kreimer, \emph{Association of multiple zeta values with
  positive knots via {F}eynman diagrams up to 9 loops}, Phys. Lett.~B
  \textbf{393} (1997), no.~3-4, 403--412.

\bibitem[Bor77]{borel}
A.~Borel, \emph{Cohomologie de {${\rm SL}_{n}$} et valeurs de fonctions zeta
  aux points entiers}, Ann. Scuola Norm. Sup. Pisa Cl. Sci.~(4) \textbf{4}
  (1977), no.~4, 613--636.

\bibitem[BR01]{ballrivoal}
K.~Ball and T.~Rivoal, \emph{Irrationalit{\'e} d'une infinit{\'e} de valeurs de
  la fonction z\^eta aux entiers impairs}, Invent. Math. \textbf{146} (2001),
  no.~1, 193--207.

\bibitem[Bro12a]{brownMTMZ}
F.~Brown, \emph{Mixed {T}ate motives over {$\mathbb Z$}}, Ann. of Math.~(2)
  \textbf{175} (2012), no.~2, 949--976.

\bibitem[Bro12b]{browndecomposition}
\bysame, \emph{On the decomposition of motivic multiple zeta values},
  Galois-{T}eich\-m\"{u}ller theory and arithmetic geometry, Adv. Stud. Pure
  Math., vol.~63, Math. Soc. Japan, Tokyo, 2012, pp.~31--58.

\bibitem[Bro17]{brownfeynmancosmic}
\bysame, \emph{Feynman amplitudes, coaction principle, and cosmic {G}alois
  group}, Commun. Number Theory Phys. \textbf{11} (2017), no.~3, 453--556.

\bibitem[BS03]{xupszeta}
N.~Berline and C.~Sabbah (eds.), \emph{La fonction z{\^e}ta}, Journ\'ees X-UPS,
  {\'E}ditions {\'E}cole polytechnique, Palaiseau, 2003.

\bibitem[BS12]{brownschnetzK3}
F.~Brown and O.~Schnetz, \emph{A {K}3 in {$\phi^4$}}, Duke Math.~J.
  \textbf{161} (2012), no.~10, 1817--1862.

\bibitem[Car01]{cartiermad}
P.~Cartier, \emph{A mad day's work: from {G}rothendieck to {C}onnes and
  {K}ontsevich. {T}he evolution of concepts of space and symmetry}, Bull. Amer.
  Math. Soc. (N.S.) \textbf{38} (2001), no.~4, 389--408.

\bibitem[Che77]{chen}
K.-T. Chen, \emph{Iterated path integrals}, Bull. Amer. Math. Soc. (N.S.)
  \textbf{83} (1977), no.~5, 831--879.

\bibitem[Del74]{deligneweil1}
P.~Deligne, \emph{La conjecture de {W}eil.~{I}}, Publ. Math. Inst. Hautes
  {\'E}tudes Sci. (1974), no.~43, 273--307.

\bibitem[Del80]{deligneweil2}
\bysame, \emph{La conjecture de {W}eil.~{II}}, Publ. Math. Inst. Hautes
  {\'E}tudes Sci. (1980), no.~52, 137--252.

\bibitem[Del89]{delignedroiteprojective}
\bysame, \emph{Le groupe fondamental de la droite projective moins trois
  points}, Galois groups over {${\bf Q}$} ({B}erkeley, {CA}, 1987), Math. Sci.
  Res. Inst. Publ., vol.~16, Springer, New York, 1989, pp.~79--297.

\bibitem[DG05]{delignegoncharov}
P.~Deligne and A.~B. Goncharov, \emph{Groupes fondamentaux motiviques de {T}ate
  mixte}, Ann. Sci. {\'E}cole Norm. Sup.~(4) \textbf{38} (2005), no.~1, 1--56.

\bibitem[DR31]{derhamthese}
G.~De~Rham, \emph{Sur l'analysis situs des vari\'et\'es \`a $n$ dimensions},
  Doctorat d'\'etat, 1931,
  \url{http://www.numdam.org/item/THESE_1931__129__1_0}.

\bibitem[{Dri}91]{drinfeldquasihopf}
V.~G. {Drinfel'd}, \emph{{On quasitriangular quasi-{H}opf algebras and a group
  closely connected with $\text{Gal}(\overline{\mathbb Q}/\mathbb Q)$}},
  Leningrad Math.~J. \textbf{2} (1991), no.~4, 829--860.

\bibitem[Eul38]{euler1738summatione}
L.~Euler, \emph{De summatione innumerabilium progressionum}, Commentarii
  academiae scientiarum Petropolitanae (1738), 91--105.

\bibitem[Eul40]{euler1740summis}
\bysame, \emph{De summis serierum reciprocarum}, Commentarii academiae
  scientiarum Petropolitanae (1740), 123--134.

\bibitem[Fis04]{fischlerbourbaki}
S.~Fischler, \emph{Irrationalit{\'e} de valeurs de z\^{e}ta (d'apr\`es
  {A}p{\'e}ry, {R}ivoal,...)}, {Séminaire Bourbaki}, Ast{\'e}risque, vol. 294,
  Soci{\'e}t{\'e} Math{\'e}matique de France, Paris, 2004, pp.~27--62.

\bibitem[GKZ06]{ganglkanekozagier}
H.~Gangl, M.~Kaneko, and D.~Zagier, \emph{Double zeta values and modular
  forms}, Automorphic forms and zeta functions, World Sci. Publ., Hackensack,
  NJ, 2006, pp.~71--106.

\bibitem[Gon01a]{goncharovdihedral}
A.~B. Goncharov, \emph{The dihedral {L}ie algebras and {G}alois symmetries of
  {$\pi_1^{(l)}(\mathbb P^1-(\{0,\infty\}\cup\mu_N))$}}, Duke Math.~J.
  \textbf{110} (2001), no.~3, 397--487.

\bibitem[Gon01b]{goncharovmultiplepolylogs}
A.~B. Goncharov, \emph{Multiple polylogarithms and mixed {T}ate motives}, 2001,
  \eprint{math/0103059}.

\bibitem[Gon01c]{goncharovmzvmodular}
A.~B. Goncharov, \emph{Multiple $\zeta$-values, {G}alois groups, and geometry
  of modular varieties}, European Congress of Mathematics (Barcelona, July
  10--14, 2000), Vol.~I, Birkh{\"a}user Basel, Basel, 2001, pp.~361--392.

\bibitem[Gon05]{goncharovgaloissym}
\bysame, \emph{Galois symmetries of fundamental groupoids and noncommutative
  geometry}, Duke Math.~J. \textbf{128} (2005), no.~2, 209--284.

\bibitem[Gro66]{grothendieckderham}
A.~Grothendieck, \emph{On the de {R}ham cohomology of algebraic varieties},
  Publ. Math. Inst. Hautes {\'E}tudes Sci. (1966), no.~29, 95--103.

\bibitem[Gro68]{grothendiecklefschetz}
\bysame, \emph{Formule de {L}efschetz et rationalit{\'e} des fonctions {$L$}},
  Dix expos{\'e}s sur la cohomologie des sch{\'e}mas, Adv. Stud. Pure Math.,
  vol.~3, North-Holland, Amsterdam, 1968, pp.~31--45.

\bibitem[Hai86]{hainhomotopy}
R.~M. Hain, \emph{Mixed {H}odge structures on homotopy groups}, Bull. Amer.
  Math. Soc. (N.S.) \textbf{14} (1986), no.~1, 111--114.

\bibitem[Hai87a]{hainderham}
\bysame, \emph{The de {R}ham homotopy theory of complex algebraic
  varieties.~{I}}, $K$-Theory \textbf{1} (1987), no.~3, 271--324.

\bibitem[Hai87b]{haingeometry}
\bysame, \emph{The geometry of the mixed {H}odge structure on the fundamental
  group}, Algebraic geometry, {B}owdoin, 1985 ({B}runswick, {M}aine, 1985),
  Proc. Sympos. Pure Math., vol.~46, American Mathematical Society, Providence,
  RI, 1987, pp.~247--282.

\bibitem[HJPO99]{minhetal}
{Hoang Ngoc Minh}, G.~Jacob, M.~Petitot, and N.~E. Oussous, \emph{Aspects
  combinatoires des polylogarithmes et des sommes d'{E}uler-{Z}agier}, S{\'e}m.
  Lothar. Combin. \textbf{43} (1999), Article ID B43e (29 pages).

\bibitem[HMS17]{hubermullerstachbook}
A.~Huber and S.~M\"{u}ller-Stach, \emph{Periods and {N}ori motives}, Ergeb.
  Math. Grenzgeb.~(3), vol.~65, Springer, Cham, 2017.

\bibitem[Hof97]{hoffman}
M.~E. Hoffman, \emph{The algebra of multiple harmonic series}, J.~Algebra
  \textbf{194} (1997), no.~2, 477--495.

\bibitem[Hof00]{hoffmanquasishuffle}
\bysame, \emph{Quasi-shuffle products}, J.~Algebraic Combin. \textbf{11}
  (2000), no.~1, 49--68.

\bibitem[HZ87]{hainzucker}
R.~M. Hain and S.~Zucker, \emph{Unipotent variations of mixed {H}odge
  structure}, Invent. Math. \textbf{88} (1987), no.~1, 83--124.

\bibitem[IKZ06]{iharakanekozagier}
K.~Ihara, M.~Kaneko, and D.~Zagier, \emph{Derivation and double shuffle
  relations for multiple zeta values}, Compositio Math. \textbf{142} (2006),
  no.~2, 307--338.

\bibitem[KZ01]{kontsevichzagier}
M.~Kontsevich and D.~Zagier, \emph{Periods}, Mathematics unlimited---2001 and
  beyond, Springer, Berlin, 2001, pp.~771--808.

\bibitem[Lin82]{lindemann}
F.~Lindemann, \emph{Ueber die {Z}ahl $\pi$}, Math. Ann. \textbf{20} (1882),
  no.~2, 213--225.

\bibitem[PS16]{panzerschnetz}
E.~Panzer and O.~Schnetz, \emph{The {G}alois coaction on {$\phi^4$} periods},
  Commun. Number Theory Phys. \textbf{11} (2016), 657--705.

\bibitem[Reu93]{reutenauerfreelie}
C.~Reutenauer, \emph{Free {L}ie algebras}, London Math. Soc. Monographs,
  vol.~7, The Clarendon Press, Oxford University Press, New York, 1993.

\bibitem[Riv00]{rivoalcras}
T.~Rivoal, \emph{La fonction z\^{e}ta de {R}iemann prend une infinit{\'e} de
  valeurs irrationnelles aux entiers impairs}, C.~R. Acad. Sci. Paris S{\'e}r.
  I Math. \textbf{331} (2000), no.~4, 267--270.

\bibitem[Rud87]{rudin}
W.~Rudin, \emph{Real and complex analysis}, 3\ieme ed., McGraw-Hill Book Co.,
  New York, 1987.

\bibitem[Sou10]{souderes}
I.~Soud\`{e}res, \emph{Motivic double shuffle}, Int.~J. Number Theory
  \textbf{6} (2010), no.~2, 339--370.

\bibitem[Sti30]{stirling}
J.~Stirling, \emph{Methodus differentialis, sive tractatus de summatione et
  interpolatione serierum infinitarum}, G. Bowyer, impensis G. Strahan
  (Londini), 1730.

\bibitem[Ter02]{terasoma}
T.~Terasoma, \emph{Mixed {T}ate motives and multiple zeta values}, Invent.
  Math. \textbf{149} (2002), no.~2, 339--369.

\bibitem[Voe00]{voevodskytriangulated}
V.~Voevodsky, \emph{Triangulated categories of motives over a field}, Cycles,
  transfers, and motivic homology theories, Ann. of Math. Stud., vol. 143,
  Princeton Univ. Press, Princeton, NJ, 2000, pp.~188--238.

\bibitem[Wei49]{weilconjectures}
A.~Weil, \emph{Numbers of solutions of equations in finite fields}, Bull. Amer.
  Math. Soc. \textbf{55} (1949), 497--508.

\bibitem[Zag94]{zagiervalues}
D.~Zagier, \emph{Values of zeta functions and their applications}, First
  European Congress of Mathematics (Paris, July 6--10, 1992) Vol.~II,
  Birkh{\"a}user Basel, Basel, 1994, pp.~497--512.

\bibitem[Zud01]{zudilinfour}
W.~Zudilin, \emph{One of the numbers {$\zeta(5)$}, {$\zeta(7)$}, {$\zeta(9)$},
  {$\zeta(11)$} is irrational}, Uspehi Mat. Nauk \textbf{56} (2001),
  no.~4(340), 149--150.

\end{thebibliography}
		
\end{document}